\documentclass[a4paper, 12pt]{article}

\usepackage{
amsfonts,amssymb,amsthm,amsmath}

\newcommand{\diag}{\operatorname{diag}}

\newcommand{\im}{\operatorname{Im}}
\newcommand{\inter}{\operatorname{int}}
\newcommand{\ind}{\operatorname{ind}}

\newcommand{\Log}{\operatorname{Log}}
\newcommand{\spa}{\operatorname{span}}

\newtheorem{thm}{Theorem}[section]
 \newtheorem{cor}[thm]{Corollary}
 \newtheorem{lem}[thm]{Lemma}
 \newtheorem{prop}[thm]{Proposition}
 \theoremstyle{definition}
 \newtheorem{defn}[thm]{Definition}
 \theoremstyle{remark}
 \newtheorem{rem}[thm]{Remark}
 
 \numberwithin{equation}{section}

\begin{document}

\title{\renewcommand{\baselinestretch}{0.7}Factorization in a torus \\ and Riemann-Hilbert problems\\ \vspace{1.2cm}\large{M. C. C\^amara{\small{*}} \; and \; M.T.Malheiro{\small{**}}}\\ \vspace{12.5cm}
\begin{flushleft}
\begin{tabular}{c}
  \hline
\hspace{5cm}  \\
\end{tabular}\\
\vspace{-0.3cm}\hspace{0.5cm}{\small{*} Departamento \,de\,
Matem\'atica,\, Instituto \,Superior \,T\'ecnico, Universidade \;
T\'ecnica \; de\; Lisboa. \;
 Email\;
address cristina.camara@math.ist.utl.pt}\\
\hspace{0.5cm}{\small{**} Centro de Matem\'{a}tica, Departamento de Matem\'{a}tica e Aplica\c{c}\~{o}es, \ 
Universidade do Minho, Campus de Azur\'{e}m, 4800-058 Guimar\~{a}es,\; Email \; address mtm@math.uminho.pt}
\end{flushleft}}

\date{}

\maketitle

\begin{abstract}
A  factorization, which is shown to be a generalization of Wiener-Hopf factorization, is studied for H\"{o}lder continuous functions defined on a contour $\Gamma$ that is the pullback of $\dot{\mathbb{R}}$ (or the unit circle) on a Riemann surface $\Sigma$ of genus 1. The existence of a holomorphic $\Sigma$-factorization for every invertible function in that class is established and formulas  are given for the factors. A new concept of meromorphic $\Sigma$-factorization is introduced and studied, and its relation with holomorphic $\Sigma$-factorization is discussed. This is applied to study and solve some vectorial Riemann Hilbert problems, including Wiener-Hopf matrix factorization, as well as to study some properties of Toeplitz operators with $2 \times 2$ matrix symbols.

\end{abstract}


\section{Introduction}
Let $C_{\mu}(\dot{\mathbb{R}})$ denote the Banach algebra of functions that are continuous and satisfy a H\"{o}lder condition with exponent $\mu \in ]0,1[$ on $\dot{\mathbb{R}}$ (\cite{MP}) and let $C_{\mu}^{\pm}(\dot{\mathbb{R}}):=C_{\mu}(\dot{\mathbb{R}})\cap H_{\infty}^{\pm}$ where $H_{\infty}^{\pm}:=H^{\infty}(\mathbb{C}^{\pm})$ are the Hardy spaces of functions which are analytic and bounded in the half-planes $\mathbb{C}^{\pm}$, respectively.

Denoting by $\mathcal{GA}$ the group of invertible elements in an algebra $\mathcal{A}$, it is well known that any $f\in \mathcal{G}C_{\mu}(\dot{\mathbb{R}})$ can be represented as a product
\begin{equation}\label{1.1}
f=f_-\mathfrak{r}^kf_+
\end{equation}
where $f_-^{\pm 1} \in C^-_{\mu}(\dot{\mathbb{R}}), \quad f_+^{\pm 1} \in C^+_{\mu}(\dot{\mathbb{R}}),$ $\mathfrak{r}(\xi)=\frac{\xi-i}{\xi+i}, \quad \xi \in \mathbb{R},$
and $k \in \mathbb{Z}$ is the index of the complex function $f$ relative to the origin, $k=\nolinebreak[4]\ind f$. The representation (\ref{1.1}) is called a \emph{Wiener-Hopf} (or \emph{WH}) \emph{factorization} of $f$;  if $\ind f=0$, which is equivalent to having $\log f \in C_{\mu}(\dot{\mathbb{R}})$ (\cite{CG}), we can write
\begin{equation}\label{1.3A}
f=f_-f_+
\end{equation}
and the WH factorization is said to be \emph{canonical}. It allows to solve Riemann-Hilbert (RH) problems of the form
\begin{equation}\label{1.4}
f\varphi_+=\varphi_-+\psi,
\end{equation}
where $\psi$ is a given function and the unknowns $\varphi_{\pm}$ belong to certain spaces of functions analytic in $\mathbb{C}^{\pm}$, respectively. The representation (\ref{1.1}) is also important in the study of several properties of Toeplitz operators
\begin{equation*}\label{1.5}
T_f:H_p^+\rightarrow H_p^+, \quad T_f\varphi_+=P^+(f\varphi_+)
\end{equation*}
where, for $p \in ]1,+\infty[$, we denote by $H_p^+$ the Hardy space $H^p(\mathbb{C}^+)$ (\cite{Duren}) and by $P^+$ we denote the projection of $L_p(\mathbb{R})$ onto $H_p^+$ parallel to $H_p^-:=H^p(\mathbb{C}^-)$, identifying $H_p^{\pm}$ with subspaces of $L_p(\mathbb{R})$ (\cite{MP,CG,BS,GK,LS}).

Many problems in engineering, physics and mathematics also lead to RH problems with matrix coefficients, for which a factorization analogous to (\ref{1.1}) can be defined and used in a similar way. However, contrary to the scalar case, methods to obtain its factors are known only for particular cases, even in what can be considered the simplest non-scalar case, that of $2\times 2$ matrix functions (\cite{MP,CG,LS}).

In this case, though, it was shown in \cite{CSS} that for every $2 \times 2$ matrix function $G$ with entries in $C_{\mu}(\dot{\mathbb{R}})$ and possessing an inverse of the same type, there are symmetric matrix functions $Q_1 \in \mathcal{G}(C^-_{\mu}(\dot{\mathbb{R}})+\mathcal{R})^{2 \times 2}, \ \ Q_2 \in \mathcal{G}(C^+_{\mu}(\dot{\mathbb{R}})+\mathcal{R})^{2 \times 2}$
 (where by $\mathcal{R}$ we denote the space of rational functions with poles off $\dot{\mathbb{R}}$) such that
\begin{equation}\label{1.6}
G^TQ_1G=\det G. \ Q_2.
\end{equation}
It is shown moreover that $Q_1$ and $Q_2$ can be chosen such that $\det Q_1=\det Q_2=\mathfrak{p}$, where $\mathfrak{p}$ is a monic polynomial admitting, at most, simple zeros. Denoting by $C(Q_1,Q_2)$ the class of all matrix functions $G$ satisfying (\ref{1.6}) for a given pair $(Q_1,Q_2)$, it is then possible to associate with each class $C(Q_1,Q_2)$ a Riemann surface $\Sigma$ defined by an algebraic curve of the form $\tau^2=\mathfrak{p}(\xi)$.
This, in its turn, allows to reduce the factorization problem for a large class of $2\times 2$ matrix functions to a scalar RH problem in $\Sigma$, thus providing a general framework that goes significantly beyond the partial results that could previously be found in the literature (for general references on RH problems in Riemann surfaces, including their relations with the factorization of particular types of matrix functions see, for instance, \cite{Rodin,Springer,Zverovich1} and, more recently, \cite{Zverovich3} and references in it).

As a tool that can be considered naturally suggested by the use of (\ref{1.1}) to solve scalar RH problems relative to $\dot{\mathbb{R}}$ (or, equivalently, relative to the unit circle, as it is often the case), the concept of \emph{$\Sigma$-factorization} of a function $f$ defined on a contour $\Gamma$, which is the pullback of $\dot{\mathbb{R}}$ (or the unit circle) in $\Sigma$, is also introduced in \cite{CSS}. This factorization takes the form
\begin{equation}\label{1.8}
f=f_-rf_+
\end{equation}
where $f_{\pm}$ as well as their inverses belong to certain spaces of analytic functions in $\Sigma^{\pm}$, respectively, denoting by $\Sigma^{\pm}$ the pullback of $\mathbb{C}^{\pm}$ in $\Sigma$. If $r=1$ in (\ref{1.8}), we say that is a \emph{special $\Sigma$-factorization}.

It is shown in \cite{CSS} that a representation (\ref{1.8}) exists for all $f$ satisfying a H\"{o}lder condition with exponent $\mu \in ]0,1[$ on $\Gamma$ ($f \in C_{\mu}(\Gamma)$) and such that we have $\log f \in C_{\mu}(\Gamma)$. In contrast with the analogous situation in $C_{\mu}(\dot{\mathbb{R}})$, in this case $f$ does not possess a special $\Sigma$-factorization (which can be considered as the natural analogue of (\ref{1.3A}) in $\Sigma$), unless some additional and rather restrictive condition is satisfied. Assuming that $\log f \in C_{\mu}(\Gamma)$ (no conditions for existence of a $\Sigma$-factorization having been established otherwise), a method is proposed in \cite{CSS} to obtain (\ref{1.8}). Its application in the case of surfaces with genus greater than 1, however, presents great difficulties. Even in the case of genus  1, some questions naturally arise regarding  the formulas defining the factors $f_{\pm}$ and the form of the meromorphic middle factor  $r$ in (\ref{1.8}). Namely, the latter is given as a power of order $N \in \mathbb{N}$ of a rational function defined in terms of Riemann theta functions and depending also on $N$, where $N$ is large enough in a given sense (see Theorems 3.4 and 3.10 in \cite{CSS}). Since $f_{\pm}$ also depend on $N$ and we can replace $N$ by any $\tilde{N}>N$, it is clear that a factorization obtained by using the method proposed in \cite{CSS} is highly non-unique, unless it is a special $\Sigma$-factorization, and can present unnecessary difficulties.

Defining an appropriate form for the rational factor $r$ in (\ref{1.8}) is particularly important. On the one hand, since a $\Sigma$-factorization of $f$ is applied to solve RH problems of the form (\ref{1.4}) in $\Sigma$ in a way which is similar to that used when applying a WH factorization to solve RH problems in $\mathbb{C}$ (cf. \cite{CSS}), $r$ should be of a simple form and in particular $N$ should have the smallest possible value. On the other hand, this form should by itself provide some information on the RH problem with coefficient $f$ (such as the dimension of the space of solutions to the homogeneous RH problem), or on the Toeplitz operator $T_G$, if $G$ is a $2 \times 2$ matrix symbol whose factorization problem can be reduced to a scalar RH problem in $\Sigma$ as described in \cite{CSS}.

Our main purpose in this paper is to develop the study of $\Sigma$-factorization for functions defined on a contour $\Gamma$ in a Riemann surface $\Sigma$ of genus 1, and address the questions and difficulties that were mentioned above, considering its applications to the study of the solvability of RH problems in $\Sigma$ and of some properties of Toeplitz operators with $2 \times 2$ matrix symbols, such as invertibility or the characterization of its kernel and cokernel.

We assume here that $\Gamma$ is the pullback of $\dot{\mathbb{R}}$ in $\Sigma$ having in mind applications to problems which are originally formulated in $\dot{\mathbb{R}}$. All the results can however be translated to the case where, instead of $\dot{\mathbb{R}}$, the unit circle is the natural domain to be considered. We assume moreover that $\Sigma$ is given in a standard form which avoids computational difficulties (\cite{Erdelyi}) and allows to use convenient analogues of the Cauchy kernel (\cite{Zverovich1, Zverovich2}). The main results concerning $\Sigma$-factorization are stated in sections 4 and 5, while in section 6 we illustrate their application to the study of vectorial RH problems, WH factorization and some properties of Toeplitz operators with $2 \times 2$ symbols. Sections 2 and 3 can be considered of an auxiliary nature. The paper is organized as follows.

In section 2 we settle some notation and recall several preliminary results that will be needed later.

In section 3 we define and study the properties of functions of a certain form which are meromorphic in $\Sigma^+$ or in $\Sigma^-$ (recall that $\Sigma^{\pm}$ denotes the pullback of $\mathbb{C}^{\pm}$ on $\Sigma$), or rational, which play a crucial role in the results that follow. All these functions are represented in the form $f_1+\tau f_2$ where $f_1, f_2$ can be identified with functions in $C_{\mu}(\dot{\mathbb{R}})$. In particular, rational functions are represented in the form $r_1+\tau r_2$ with $r_1, r_2 \in \mathcal{R}$. This option turns out to be crucial in simplifying the results and in obtaining truly explicit factorizations in the last section.

In section 4 we show that every $f \in \mathcal{G}C_{\mu}(\Gamma)$ admits a (holomorphic) $\Sigma$-factorization (\ref{1.8}) and we present explicit formulas for its factors, their form being particularly simple in the case of existence of a special $\Sigma$-factorization.

By introducing a new concept of meromorphic $\Sigma$-factorization in section 5, we show that it is possible to simplify the rational middle factor and reduce the number of zeros and poles that we have to deal with, when applying a factorization of $f$ to solve RH problems in $\Sigma$ with coefficient $f$. This concept of meromorphic $\Sigma$-factorization actually sheds some new light on $\Sigma$-factorization; in particular it clarifies the relation between the existence of a special $\Sigma$-factorization and the existence of an M-special $\Sigma$-factorization (see (\ref{5.4A})).

In section 6 we apply the results of the preceding sections to characterize the kernels and establish invertibility conditions for Toeplitz operators with symbols in a class of $2 \times 2$ Daniele-Khrapkov  matrices and to obtain the explicit factorization of their symbols, both in the canonical and in the non-canonical cases. Two examples are given, one of which is motivated by the problem of existence of global solutions to a Lax equation for some integrable systems (\cite{CSS1, Reyman}).
\section{Preliminary Results}
\subsection{Notations}

We start by establishing some notation regarding Riemann surfaces (for a general reference see, for instance, \cite{Springer,Miranda}).

Let $\Sigma$ be the Riemann surface of genus 1 obtained by the compactification of the elliptic algebraic curve $\Sigma_{0}=\{(\xi, \tau)\in \mathbb{C}^2: \tau^2=\mathfrak{p}(\xi)\}$
where we assume the polynomial equation defining $\Sigma_{0}$ to be, up to a simple change of variables, in Legendre's normal form
\begin{equation}
\tau^2=(1+\xi^2)(k^2_{0}+\xi^2), \quad k_{0}>1
\label{2.2}
\end{equation}
(\cite{Erdelyi}), by adding two points "at infinity". In a neighbourhood of these, $\xi^{-1}$ is taken as the local parameter. We take the meromorphic function $(\xi, \tau)\mapsto \tau$ as the local parameter in a neighbourhood of the branch points; at all other points, $\xi$ is the local parameter.

It is convenient to view $\Sigma$ as a two-sheeted covering of $\mathbb{C}_{\infty}=\mathbb{C}\cup \{\infty\}$ with branch cuts $[-ik_{0}, -i]$ and $[i, ik_{0}]$, (using the notation $[a,b]$, $[a,b[$ and so on for line segments, including or excluding the endpoints, oriented from $a$ to $b$ when the orientation is relevant) via the meromorphic function
\begin{equation*}
\begin{array}{ccccccc}
  \Pi : & \Sigma & \rightarrow & \mathbb{C}_{\infty}, & (\xi, \tau) & \mapsto & \xi .
\end{array}
\label{2.3}
\end{equation*}
We say that $\xi$ is the \emph{projection} of $(\xi, \tau)$ in $\mathbb{C}_{\infty}$ or, equivalently, that $(\xi, \tau)$ is a \emph{preimage} of $\xi$ in $\Sigma$. Denoting by $\rho$ the branch of $\sqrt{\mathfrak{p}}$ (where $\mathfrak{p}(\xi)$ is defined by the right-hand side of (\ref{2.2})) for which Re $\rho \geq 0$, the points $(\xi, \rho(\xi))$ (resp. $(\xi, -\rho(\xi))$) are in the upper (resp. lower) sheet $\Sigma_{1}$ (resp. $\Sigma_{2}$ ) and we denote by  $\mathbf{\xi_{1}}$, $\mathbf{\xi_{2}}$ the preimages of $\xi$ in $\Sigma_{1}$ and $\Sigma_{2}$ respectively.

By $\Sigma^{\pm}$ we denote the inverse images under $\Pi$ of $\mathbb{C}^{\pm}$, respectively, and by $\Gamma$ the pullback of the compactified real line $\dot{\mathbb{R}}$. Note that $\Pi^{-1}(\dot{\mathbb{R}})$ consists of two disjoint closed paths (whose orientation is induced by that of the real line) $\Gamma_{1}$ and $\Gamma_{2}$ in $\Sigma_{1}$ and $\Sigma_{2}$ respectively, dividing $\Sigma$ into the two disjoint regions $\Sigma^{+}$ and $\Sigma^{-}$.

Denoting by * the hyperelliptic involution defined in $\Sigma$ by \linebreak $(\xi,\tau)\mapsto(\xi, -\tau)$, we will also use the following notations: $D_*=*(D)$ for $D\subset \Sigma$, $f_*$ for the composition of a complex valued function $f$, defined in a *-invariant subset $D(=D_*)$ of $\Sigma$, with $*: f_*=f\circ *$. Any function in a *-invariant subset of $\Sigma$ can be decomposed uniquely in the form $f=f_{\mathcal{E}}+\tau f_{\mathcal{O}}$
where $f_{\mathcal{E}}=\frac{1}{2}(f+f_*), \quad  f_{\mathcal{O}}=\frac{1}{2\tau}(f-f_*)$.

If $F$ is a complex valued function defined in $\Pi (D)$, where $D\subset \Sigma$, then we define $F_{\Pi}=F\circ\Pi :D\subset\Sigma\rightarrow \mathbb{C}$. $F_{\Pi}$ is meromorphic (resp. analytic) if $F$ is meromorphic (resp. analytic) in the corresponding domains. If $D=D_*$, then $(F_{\Pi})_*=F_{\Pi}$; conversely, if $f_*=f$ then there is a unique function $F$ in $\Pi(D)$ such that $f=F_{\Pi}$. Thus, we identify each *-invariant function $f$ in $D\subset\Sigma$ with $F$ (in $\Pi(D)\subset \mathbb{C}$) such that $f=F_{\Pi}$ and we use the same notation for both.

With this convention, if $f$ belongs to the space of H\"{o}lder continuous functions with exponent $\mu \in ]0,1[$ on $\Gamma$, denoted by $C_{\mu}(\Gamma)$, then $f_{\mathcal{E}}$, $f_{\mathcal{O}}$ and $\lambda_{+}^2f_{\mathcal{O}}$,  with
\begin{equation}
\lambda_{+}(\xi)=\xi+i,
\label{2.6}
\end{equation}belong to $ C_{\mu}(\dot{\mathbb{R}})$.

We denote by $C^{\pm}_{\mu}(\Gamma)$ (resp. $\mathcal{M}(\Sigma^{\pm})$) the subspace of $C_{\mu}(\Gamma)$ whose elements admit an analytic (resp. meromorphic) extension to $\Sigma^{\pm}$ and by $\mathcal{R}(\Sigma)$ the field of rational functions in $\Sigma$ without poles on $\Gamma$.

We will also need the Abel-Jacobi map
\begin{equation*}
A_{J} :  \Sigma  \longrightarrow  \mathbb{C}/L, \quad  A_{J}(P)=\frac{k_{0}}{i}\int_{\mathbf{0_{1}}}^P \frac{d\xi}{\tau} \ \mod L.
  \label{2.9}
\end{equation*}
where $L$ is the lattice $L=\mathbb{Z}.4K+\mathbb{Z}.2iK'$,
$K$ and $K'$ being the complete elliptic integrals (\cite{Erdelyi, Akhiezer})
\begin{equation*}
K=\int^{1}_{0}\frac{d\xi}{\sqrt{(1-\xi^2)(1-\frac{\xi^2}{k^2_{0}})}}, \quad K'=\int^{k_{0}}_{1}\frac{d\xi}{\sqrt{(\xi^2-1)(1-\frac{\xi^2}{k^2_{0}})}}.
\label{2.8}
\end{equation*}
Denoting by $\mathcal{P}$ the rectangle $\mathcal{P}=\{ s+it: s \in [-2K, 2K], t\in [-K',K']\}$
with four sides $s_{1} =[-2K+iK', 2K+iK'],\quad \gamma_{2}= [2K-iK',2K+iK']$,  $s'_{1} =[-2K-iK', 2K-iK'],\quad \gamma'_{2}= [-2K-iK',-2K+iK']$,
by the standard identified polygon representation (\cite{Miranda, Akhiezer}) the torus $\mathbb{C}/L$ is obtained from $\mathcal{P}$ by identifying the sides $s_{1}$ with $s'_{1}$ and $\gamma_{2}$ with $\gamma'_{2}$. In this representation all four vertices of $\mathcal{P}$ correspond to one point of $\mathbb{C}/L$ and the (oriented) sides $s_{1}$ and $\gamma_{2}$ correspond to closed paths: $\Pi_{L}(s_{1})=\Pi_{L}(s'_{1})$ and $\Pi_{L}(\gamma_{2})=\Pi_{L}(\gamma'_{2})$, respectively, where $\Pi_{L}:\mathbb{C}\longrightarrow \mathbb{C}/L$
is the canonical map.

Let $\sigma = A^{-1}_{J}\circ \Pi_{L}:\mathbb{C}\longrightarrow \Sigma$. We remark that, defining $\widetilde{\mathcal{P}}=\mathcal{P}\backslash (s_{1}'\cup \gamma'_{2})$, $\sigma_{|\widetilde{\mathcal{P}}}$ is a bijective map. We will use the following notation:
\begin{equation*}
\hspace{-0.5cm}\gamma_{1} = [iK',-iK'] \quad  (\gamma_{1}=-\gamma_{2}-2K, \text{ and } \sigma (\gamma_{j})=\Gamma_{j} \text{ for } j=1,2)
\label{2.15}
\end{equation*}
\begin{equation*}
\hspace{-2.8cm}\Omega^{+}= \mathcal{P}\cap \{z \in \mathbb{C}: \text{Re}(z) \in ]0, 2K[ \} \quad (\sigma (\Omega^{+})=\Sigma^{+})
\label{2.16}
\end{equation*}
\begin{equation*}
\hspace{-2.3cm}\Omega^{-}= \mathcal{P}\cap \{z \in \mathbb{C}: \text{Re}(z) \in ]-2K, 0[ \} \quad (\sigma (\Omega^{-})=\Sigma^{-}).
\label{2.17}
\end{equation*}
Let moreover $\mathcal{A}$ denote the closed path on $\Sigma$,
\begin{equation}
\mathcal{A}=\sigma ([-2K, 2K])
\label{2.18}
\end{equation}
whose projection on $\mathbb{C}_{\infty}$ is the line segment $[-i,i]$.
We have \linebreak $\int_{\mathcal{A}}(d\xi/\tau)=4iK/k_{0}, \quad \int_{\Gamma_{1}}(d\xi/\tau)=2K'/k_{0}.$

\subsection{Singular Integral Operators}

We introduce now some integrals of Cauchy type and present their fundamental properties. They are defined making use of analogues of the Cauchy kernel, of the same type as those constructed in (\cite{CSS,Zverovich1,Zverovich2}), taking here into account that the points corresponding to $\infty$ belong to $\Gamma$.
\begin{defn}
For $f \in C_{\mu}(\Gamma)$, let
\begin{equation*}
P^{\pm}_{\Gamma}f(\xi,\tau)=\pm \frac{1}{4\pi i}\left[(\xi+i)\int_{\Gamma}\frac{f(\xi_{0},\tau_{0})}{\xi_{0}+i}\frac{d\xi_{0}}{\xi_{0}-\xi}+\right.
\end{equation*}
\begin{equation*}
\left.\frac{\tau}{\xi+i}\int_{\Gamma}\frac{(\xi_{0}+i)f(\xi_{0},\tau_{0})}{\tau_{0}}\frac{d\xi_{0}}{\xi_{0}-\xi}\right]
\label{2.20}
\end{equation*}
where the integrals are understood in the sense of Cauchy's principal value.
\end{defn}
Denoting by $\widetilde{P}^{\pm}_{\mathbb{R}}$ the projections defined in $C_{\mu}(\dot{\mathbb{R}})$ by $\widetilde{P}^{\pm}_{\mathbb{R}} f=\lambda_{+}P^{\pm}_{\mathbb{R}}(\lambda_{+}^{-1}f)$
where $\lambda_{+}$ is defined in (\ref{2.6}) and $P^{\pm}_{\mathbb{R}}$ are the projections associated with the singular integral operator with Cauchy kernel $S_{\mathbb{R}}$ (\cite{MP}), i.e., $P^{\pm}_{\mathbb{R}}=(1/2)(I\pm S_{\mathbb{R}})$, it is easy to see that
\begin{equation}
P^{\pm}_{\Gamma} f=\widetilde{P}^{\pm}_{\mathbb{R}} f_{\mathcal{E}}+\tau\lambda_{+}^{-2}\widetilde{P}^{\pm}_{\mathbb{R}} (\lambda_{+}^2f_{\mathcal{O}}).
\label{2.22}
\end{equation}
From (\ref{2.22}) and the properties of $S_{\mathbb{R}}$ and $\widetilde{P}^{\pm}_{\mathbb{R}}$, it is clear that $P^{\pm}_{\Gamma}$ are bounded operators in $C_{\mu}(\Gamma)$ and the following holds (\cite{CSS}):
\begin{prop}
\begin{itemize}
               \item[(i)] $P^{\pm}_{\Gamma}$ are complementary projections in $C_{\mu}(\Gamma)$.
               \item[(ii)] $\im P^{+}_{\Gamma}=C_{\mu}^+(\Gamma)$.
               \item[(iii)] $\im P^{-}_{\Gamma}=C_{\mu}^-(\Gamma) \oplus \spa \{\tau\lambda_{+}^{-1}\}$.
               \item[(iv)] $P^{-}_{\Gamma}f \in C_{\mu}^-(\Gamma)$ if and only if
               \begin{equation}
                \int_{\Gamma}\frac{f(\xi_{0},\tau_{0})}{\tau_{0}}d\xi_{0}=0 \ ;
                \label{2.23}
                \end{equation}
                otherwise, $P^{-}_{\Gamma}f$ is meromorphic in $\Sigma^{-}$, with a simple pole at the branch point $-i$.
                \item[(v)] Every function $f \in C_{\mu}(\Gamma)$ admits a decomposition
                \begin{equation*}
               f=P^{+}_{\Gamma}f+P^{-}_{\Gamma}f=P^{+}_{\Gamma}f+f_{-}- \frac{K}{k_{0}\pi}\alpha_{f}\tau\lambda_{+}^{-1}
                \label{2.24}
                \end{equation*}
                where $f_{-}\in C_{\mu}^-(\Gamma)$ and
                \begin{equation}
               \alpha_{f}=\frac{k_{0}}{4Ki}\int_{\Gamma}\frac{f(\xi_{0}, \tau_{0})}{\tau_{0}}d\xi_{0}.
                \label{2.25}
                \end{equation}
             \end{itemize}
\end{prop}
In what follows we will need two other integrals of Cauchy type using the Behnke-Stein analogue of the Cauchy kernel (\cite{Zverovich2}):
\begin{defn}
For $f \in C_{\mu}(\Gamma)$, let
\begin{equation}
\widetilde{P}^{\pm}_{\Gamma}f(\xi,\tau)=P^{\pm}_{\Gamma}f(\xi,\tau)\mp \frac{\alpha_{f}}{2\pi i}\frac{\tau}{\xi+i}\int_{\mathcal{A}}\frac{\xi_{0}+i}{2\tau_{0}}\frac{d\xi_{0}}{\xi_{0}-\xi}
\label{2.26}
\end{equation}
where $\alpha_{f}$ and $\mathcal{A}$ were defined in (\ref{2.25}) and (\ref{2.18}) respectively.
\end{defn}
We have $f=\widetilde{P}^{+}_{\Gamma}f+\widetilde{P}^{-}_{\Gamma}f$ where $\widetilde{P}^{\pm}_{\Gamma}f$ has an analytic extension to $\Sigma^{\pm}\backslash \mathcal{A}$, its jump across $\mathcal{A}$ being equal to $\alpha_{f}$ (\cite{Zverovich2}).
It is easy to see that $\widetilde{P}^{\pm}_{\Gamma}f=P^{\pm}_{\Gamma}f$ if and only if $\alpha_{f}=0$, i.e., (\ref{2.23}) holds and, in this case, $\widetilde{P}^{\pm}_{\Gamma}f \in C_{\mu}^{\pm}(\Gamma)$.

It is clear that if $f \in C_{\mu}(\Gamma)$ is *-invariant, and can thus be identified with a function in $C_{\mu}(\dot{\mathbb{R}})$, we have $f_{\mathcal{O}}=0$ and $\alpha_f=0$, so that from (\ref{2.22}) and (\ref{2.26}),
\begin{equation}
\widetilde{P}^{\pm}_{\Gamma}f=\widetilde{P}^{\pm}_{\mathbb{R}}f \quad \text{if } f \in C_{\mu}(\dot{\mathbb{R}}).
\label{2.26A}
\end{equation}

\section{Meromorphic functions in $\Sigma$ and $\Sigma^{\pm}$}

In this section we define and study the properties of some functions which are meromorphic in $\Sigma$ or in $\Sigma^{\pm}$ and will be used later.

For $\phi \in C_{\mu}(\Gamma)$, let $\phi_{j}=\phi_{|\Gamma_{j}},$ $j=1,2$. We define $\ind_{j}\phi=\ind \phi_{j}, \quad j=1,2,$
where $\ind \varphi$ denotes the index of a complex function $\varphi$ continuous in $\dot{\mathbb{R}}$ with $\varphi (\xi)\neq 0$ for all $\xi \in \dot{\mathbb{R}}$ (\cite{MP}).
\begin{thm}\label{theo3.4}
 Let $S \in \mathcal{R}(\Sigma)$ be defined, up to a multiplicative constant, by the principal divisor
\begin{equation}
  D_{S}(P)=\left\{\begin{array}{ccc}
                    2 &  \text{ if } &\hspace{-2.2cm}P=\sigma (-\frac{K}{5}), \\
                    -1 & \text{ if } & P=\sigma (K) \text{ or } P=\sigma (-\frac{7K}{5}), \\
                    0 & \hspace{1.1cm} \text{otherwise}.&
                  \end{array}\right.
  \label{3.8}
\end{equation}
 Then we have $\ind_{1}S=-1, \quad \ind_{2}S=0$.
\end{thm}
\begin{proof}
Following the proof of Abel's theorem in \cite{Miranda}, we have
\begin{equation*}
2\pi i \sum_{p\in \mathcal{P}}D_{S}(\sigma(p))\int_{0}^{p}dz=4K.B(\frac{dS}{S})-2iK'.A(\frac{dS}{S})
\end{equation*}
where $B(dS/S)$ and $A(dS/S)$ denote the $\Gamma_{2}$-period and the $\mathcal{S}_{1}$-period of $dS/S$ (where $\mathcal{S}_{1}=\sigma (s_{1})$),  respectively. From (\ref{3.8}) we see that
\begin{equation*}
2\pi i (-\frac{2K}{5}-K+\frac{7K}{5})=0=4K.B(\frac{dS}{S})-2iK'.A(\frac{dS}{S})
\end{equation*}
and since the $\Gamma_{2}$-period and the $\mathcal{S}_{1}$-period of $dS/S$ are integral multiples of $2\pi i$, we conclude that
\begin{equation}
B(\frac{dS}{S})=A(\frac{dS}{S})=0.
\label{3.12}
\end{equation}
On the other hand, by the residue theorem,
\begin{equation*}
\frac{1}{2\pi i}\int_{\Gamma}\frac{dS}{S}=\frac{1}{2\pi i}\left(\int_{\Gamma_{1}}\frac{dS}{S}+\int_{\Gamma_{2}}\frac{dS}{S}\right)=-1
\end{equation*}
and from (\ref{3.12}) it follows that
\begin{equation*}
\frac{1}{2\pi i}\int_{\Gamma_{1}}\frac{dS}{S}=-1, \quad \frac{1}{2\pi i}\int_{\Gamma_{2}}\frac{dS}{S}=0.
\end{equation*}
\end{proof}
In the following sections we will also need  some functions which are not rational, but merely meromorphic in an open set containing $\Sigma^+\cup\Gamma$ or $\Sigma^-\cup\Gamma$.
Let $\rho_{+}=\sqrt{(\xi+i)(\xi+ik_{0})}$ denote the branch of the square-root which is analytic in the complex plane cut along $[-i,-ik_{0}]$ and takes the value $i\sqrt{k_{0}}$ for $\xi=0$. Analogously, let $\rho_{-}=\sqrt{(\xi-i)(\xi-ik_{0})}$ denote the branch which is analytic in the complex plane cut along $[i,ik_{0}]$ and takes the value $-i\sqrt{k_{0}}$ for $\xi=0$. We have $\rho=\rho_{-}\rho_{+}$. Let moreover $\alpha_{+}, \alpha_{-}$ be the functions defined by
\begin{equation}
\alpha_{+}(\xi,\tau)=C+\frac{\tau}{(\xi-i)\rho_{+}}, \quad \alpha_{-}(\xi,\tau)=C+\frac{\tau}{(\xi+i)\rho_{-}}
\label{3.19}
\end{equation}
where
\begin{equation}
C=\sqrt{\frac{1+k_{0}}{2}}>0.
\label{3.20}
\end{equation}
We remark, for future reference, that
\begin{equation}
C^2-1=-(C^2-k_{0})=\frac{k_{0}-1}{2}
\label{3.20A}
\end{equation}
and, for $\alpha_{\pm}$ defined by (\ref{3.19}),
\begin{equation}
\alpha_{+}(\alpha_{+})_*=\frac{k_{0}-1}{2}\frac{\lambda_{+}}{\lambda_{-}}, \quad \alpha_{-}(\alpha_{-})_*=\frac{k_{0}-1}{2}\frac{\lambda_{-}}{\lambda_{+}}, \label{3.20B}
\end{equation}
where
\begin{equation}
\lambda_{\pm}(\xi)=\xi \pm i. \label{3.20C}
\end{equation}
These functions have moreover the following properties.
\begin{thm}\label{theo3.6}
For $\alpha_{\pm}$ defined as above, we have:
\begin{itemize}
  \item[(i)] $\alpha_{+}\in \mathcal{M}(\Sigma^{+})$ with a single (simple) pole at the branch point $i$ and no zeros in $\Sigma^{+}$, and
        \begin{equation}
        \ind_{1}\alpha_{+}=0, \quad \ind_{2}\alpha_{+}=-1;
        \label{3.21}
        \end{equation}
  \item[(ii)] $\alpha_{-}\in \mathcal{M}(\Sigma^{-})$ with a single (simple) pole at the branch point $-i$ and no zeros in $\Sigma^{-}$, and
        \begin{equation}
        \ind_{1}\alpha_{-}=0, \quad \ind_{2}\alpha_{-}=1.
        \label{3.22}
        \end{equation}
\end{itemize}
\end{thm}
\begin{proof}
\begin{itemize}
  \item[(i)] It is clear that $\alpha_{+}$ is holomorphic in an open set containing \linebreak$\Sigma^{+}\cup\Gamma$, except for the branch point $i$, where it has a simple pole. It has no zeros in $\Sigma^{+}\cup\Gamma$ since from (\ref{3.20B}) $\alpha_{+}(\xi,\tau)\alpha_{+}(\xi,-\tau)=(\xi+\nolinebreak i)(\xi-\nolinebreak i)^{-1}\linebreak(k_{0}-\nolinebreak 1)/2$.

        On the other hand, identifying $\alpha_{+|\Gamma_{j}}$, $j=1,2$, with functions in $C_{\mu}(\dot{\mathbb{R}})$, we have $\alpha_{+|\Gamma_{1}}=C+\rho_{-}(\xi-i)^{-1}, \quad  \alpha_{+|\Gamma_{2}}=C-\rho_{-}(\xi-i)^{-1}$.

        The function  $\alpha_{+|\Gamma_{1}}$ is invertible in $C_{\mu}^-(\dot{\mathbb{R}})$, so that $\ind_{1}\alpha_{+}=0$, while $\alpha_{+|\Gamma_{2}}=(\xi+i)(\xi-i)^{-1}\tilde{\alpha},$  with $\tilde{\alpha}\in \mathcal{G}C_{\mu}^-(\dot{\mathbb{R}})$, so that $\ind_{2}\alpha_{+}=-1$.
  \item[(ii)] can be proved analogously.
\end{itemize}\end{proof}
\begin{cor}\label{cor3.7}
With the same assumptions as in Theorem \ref{theo3.6}, we have
\begin{equation}
\hspace{-5.6cm}\ind_{1}(\alpha_{+})_*=-1, \quad \ind_{2}(\alpha_{+})_*=0, \label{3.22A}
\end{equation}
\begin{equation}
\hspace{-5.9cm}\ind_{1}(\alpha_{-})_*=1, \quad \ind_{2}(\alpha_{-})_*=0, \label{3.22B}
\end{equation}
and
\begin{eqnarray}
     \alpha_{+}^{-1}(\alpha_{+})_* \in \mathcal{G}(C_{\mu}^+(\Gamma)),  & \ind_{1}(\alpha_{+}^{-1}(\alpha_{+})_*)=-1=-\ind_{2}(\alpha_{+}^{-1}(\alpha_{+})_*) \label{3.22C} \\
     \alpha_{-}^{-1}(\alpha_{-})_* \in \mathcal{G}(C_{\mu}^-(\Gamma)),  & \ind_{1}(\alpha_{-}^{-1}(\alpha_{-})_*)=1=-\ind_{2}(\alpha_{-}^{-1}(\alpha_{-})_*). \label{3.22D}
   \end{eqnarray}
\end{cor}
In the following theorems, we choose branches of $\log \alpha_{\pm}$ on $\Gamma$ such that $(\log \alpha_{\pm})_{|\Gamma_{1}}$ and $(\log \alpha_{\pm})_{|\Gamma_{2}\setminus \{\mathbf{\infty_{2}}\}}$ are continuous.
\begin{thm}\label{lemma3.2}With the same assumptions as in Theorem \ref{theo3.6}, we have
\begin{equation}
  \frac{k_{0}}{2\pi}\int_{\Gamma}\frac{\log \alpha_{\pm}}{\tau}d\xi =-K-iK' \mod L,
  \label{3.5}
\end{equation}
\begin{equation}
  \frac{k_{0}}{2\pi}\int_{\Gamma}\frac{\log (\alpha_{\pm})_*}{\tau}d\xi =K-iK' \mod L,
  \label{3.5B}
\end{equation}
\begin{equation}
        \frac{k_{0}}{2\pi}\int_{\Gamma}\frac{\log (\alpha_{+}^{-1}(\alpha_{+})_*)}{\tau}d\xi=2K \quad \mod L.
\label{3.24}
\end{equation}
\end{thm}
\begin{proof}
Let $D$ be an open set in $\mathcal{P}$ containing  $\inter (\Omega^{+}\cup\gamma_{1}\setminus\{\pm iK'\})$ and let $\tilde{D}$ be the simply connected domain obtained from $D$ by removing the points in the line segment $l=[K+iK',K]$. Let moreover $a_{1}=[-iK',2K-iK']$, \linebreak $b_{1}=[K+iK',iK']$, $b_{2}=b_{1}+K$.

We can define $F$ holomorphic in $\tilde{D}$ and continuous in $(\Omega^{+}\cup\partial\Omega^{+})\setminus l$ such that
\begin{eqnarray}
  F_{|\gamma_{1}} &=& \log \alpha_{+}\circ \sigma_{|\gamma_{1}} \label{3.5D} \\
  F_{|\gamma_{2}} &=& \log \alpha_{+}\circ \sigma_{|\gamma_{2}} \mod 2\pi i \mathbb{Z}\label{3.5E}
\end{eqnarray}
%
and, since $\ind_{1} \alpha_{+}=0$, $\ind_{2} \alpha_{+}=-1$,
\begin{eqnarray}
  F(z) &=& F(z-2iK'), \text{ for } z \in b_{1} \label{3.5F}\\
  F(z) &=& F(z-2iK')-2\pi i, \text{ for } z \in b_{2}. \label{3.5G}
\end{eqnarray}
%
We have then
\begin{eqnarray*}
                 0&=&\int_{\gamma_{1}}F(z)dz+\int_{a_{1}}F(z)dz+\int_{\gamma_{2}}F(z)dz+ \\
                 & & +\int_{b_{2}}F(z)dz+2\pi i\int_{K}^{K+iK'}dz+\int_{b_{1}}F(z)dz
               \end{eqnarray*}
so that, taking (\ref{3.5E}), (\ref{3.5F}) and (\ref{3.5G}) into account, we obtain the equality (\ref{3.5}) for $\alpha_+$.
The other equalities can be deduced analogously.
\end{proof}
Finally, we define and study some properties of a rational function $r_{\nu}$,
\begin{equation}
\hspace{-0.6cm}r_{\nu}(\xi,\tau)=\nu +\frac{\tau}{(\xi+i)(\xi-ik_{0})},
\label{3.16}
\end{equation} with  $\nu \in \mathbb{C}$ defined, for each value of $\beta \in \mathcal{P}_{1}\setminus \{0\}$ where
\begin{equation}
\mathcal{P}_{1}=\left\{s+it: s \in ]-K,K[, \ t \in ]-iK',iK'] \right\}
\label{3.13}
\end{equation}
by
\begin{equation}
\hspace{-2.3cm}\nu =\frac{-\tau_{0}}{(z_{0}+i)(z_{0}-ik_{0})}
\label{3.15}
\end{equation} where
\begin{equation}
\hspace{-2.3cm}(z_{0},\tau_{0})=\sigma(-K+\beta).
\label{3.14}
\end{equation}
We remark that $z_{0}$ in (\ref{3.14}) is such that $z_{0}\in \mathbb{C}^{-}$ and therefore $k_{0}z_{0}^{-1}\in \mathbb{C}^{+}$.
\begin{thm}\label{theo3.5}
The rational function $r_{\nu}$ defined above for each $\beta \in \mathcal{P}_{1}\setminus \{0\}$ has two simple poles at the branch points $-i, ik_{0},$ two simple zeros $(z_{0},\tau_{0}), \linebreak[4] (k_{0}z_{0}^{-1},k_{0}z_{0}^{-2}\tau_{0})$, no other zeros or poles, and is such that
\begin{equation}
\ind_{1}r_{\nu} =\ind_{2}r_{\nu} =0 \ ,
\label{3.17}
\end{equation}
\begin{equation}
\frac{k_{0}}{2\pi}\int_{\Gamma}\frac{\log r_{\nu}}{\tau}d\xi=\beta \ \mod L.
\label{3.18}
\end{equation}
\end{thm}
\begin{proof}
The first part of the theorem can be easily checked; (\ref{3.17}) can be verified using the same reasoning as in the proof of Theorem \ref{theo3.4}, since \linebreak$D_{r_{\nu}}=$ div$(r_{\nu})$ is given by
\begin{equation*}
  D_{r_{\nu}}(P)=\left\{\begin{array}{ccc}
                    1 & \text{ if } & P=\sigma (-K+\beta) \text{ or } P=\sigma (K+iK'-\beta)\\ \\
                    -1 & \text{ if } & \hspace{-1.1cm}P=\sigma (-K) \text{ or } P=\sigma (K+iK') \ ; \\
                  \end{array}\right.
\end{equation*}
finally (\ref{3.18}) can be obtained using the same reasoning as in the proof of
Theorem \ref{lemma3.2} taking (\ref{3.17}) into account and considering an appropriate polygon representation for the torus $\mathbb{C}\setminus L$, such that branch points do not lie on the boundary of the rectangle.
\end{proof}
We remark, for future convenience, that for $r_{0}(\xi, \tau)=\tau/[(\xi+i)(\xi-ik_{0})]$
we have $z_{0}=-ik_{0}$, $k_{0}z_{0}^{-1}=i$ and $\beta=iK'$.

An important property of the rational functions $r_{\nu}$ can be easily verified:
\begin{equation}
 \left[r_{\nu}(r_{\nu})_*\right](\xi,\tau)=(\nu^2-1)\frac{(\xi-z_{0})(\xi-\frac{k_{0}}{z_{0}})}{(\xi+i)(\xi-ik_{0})}=\nu^2-\frac{(\xi-i)(\xi+ik_{0})}{(\xi+i)(\xi-ik_{0})}.
\label{3.18A}
\end{equation}
\section{Holomorphic $\Sigma$-factorization}
A \emph{(holomorphic) $\Sigma$-factorization} of $f\in C_{\mu}(\Gamma)$ relative to $\Gamma$ is a representation of $f$ in the form
\begin{equation}
        f=f_{-}rf_{+}
\label{4.1}
\end{equation}
where $f_{\pm}\in \mathcal{G}C_{\mu}^{\pm}(\Gamma)$ and $r \in \mathcal{R}(\Sigma)$. If $r$ in (\ref{4.1}) is a constant, which can be assumed without loss of generality to be equal to 1, then (\ref{4.1}) is called a \emph{special $\Sigma$-factorization }(\cite{CSS}).

It is easy to see that for $f$ to admit a holomorphic $\Sigma$-factorization it is necessary that $f \in \mathcal{G}C_{\mu}(\Gamma)$ and that, in two special $\Sigma$-factorizations of the same function $f$, the corresponding factors must be constant multiples of each other, i.e., if $f=f_{-}f_{+}$ and $f=\tilde{f}_{-}\tilde{f}_{+}$ are special $\Sigma$-factorizations, then $\tilde{f}_{\pm}=cf_{\pm}$, with $c \in \mathbb{C}\backslash\{0\}$.

We will assume  in what follows that $f \in
\mathcal{G}C_{\mu}(\Gamma)$ and omit mentioning the contour $\Gamma$
when referring to a representation (\ref{4.1}). Moreover we start by
assuming, in the results that follows, that
$\ind_{1}f=\ind_{2}f=0$, in which case we have $\log f \in C_{\mu}(\Gamma)$ and it is known that
$f$ always admits a holomorphic $\Sigma$-factorization
(\cite{CSS}).
\begin{thm}\label{theo4.1}
Let $f \in \mathcal{G}C_{\mu}(\Gamma)$ be such that
\begin{equation}
        \ind_{1}f=\ind_{2}f=0.
\label{4.2}
\end{equation}
Then
\begin{itemize}
  \item[(i)] $f$ has a holomorphic $\Sigma$-factorization (\ref{4.1});
  \item[(ii)] $f$ admits a special $\Sigma$-factorization iff
  \begin{equation}
        \frac{k_{0}}{2\pi}\int_{\Gamma}\frac{\log f}{\tau}d\xi=4nK+2imK', \quad \text{with }m,n \in \mathbb{Z};
\label{4.3}
\end{equation}
  \item[(iii)] if (\ref{4.3}) holds, then a special $\Sigma$-factorization for $f$ is $f=f_{-}f_{+}$ with
  \begin{equation}
        f_{\pm}=\exp \left(\tilde{P}_{\Gamma}^{\pm}(\Log f)\right).
\label{4.4}
\end{equation}where \begin{equation}
        \Log f:=\log f -2im\pi.
\label{4.5}
\end{equation}
\end{itemize}
\end{thm}
\begin{proof}
(i) and (ii) were proved in \cite{CSS} and are stated here for the sake of self-containedness. As to (iii), taking (\ref{4.2}) into account, we can assume that $\log f$, as well as $\Log f$,
 are in $C_{\mu}(\Gamma)$. The jump of $\tilde{P}_{\Gamma}^{\pm}(\Log f)$ across $\mathcal{A}$ (see the paragraph before the last in section 2) is
\begin{equation*}
        \alpha_{\Log f}=\frac{k_{0}}{4Ki}\int_{\Gamma}\frac{\Log f}{\tau}d\xi=-2in\pi\ ,
\end{equation*}therefore $\exp \tilde{P}_{\Gamma}^{\pm}(\Log  f ) \in \mathcal{G}C_{\mu}^{\pm}(\Gamma)$.
On the other hand, $\Log f=\tilde{P}_{\Gamma}^{+}(\Log f)+\tilde{P}_{\Gamma}^{-}(\Log f)$
so that the factorization $f=f_{-}f_{+}$ follows, with $f_{\pm}$ defined by (\ref{4.4}) and $f_{\pm} \in \mathcal{G}C_{\mu}^{\pm}(\Gamma)$.
\end{proof}
We remark that (\ref{4.4}) provides a much simpler expression than that given in \cite{CSS} for a special $\Sigma$-factorization of $f$ satisfying the assumptions of Theorem \nolinebreak \ref{theo4.1}; not withstanding the differences in their expressions, the factors $f_{\pm}$ can only differ by a non-zero constant multiplicative factor.

The previous result also suggests that the value of
\begin{equation}
        \frac{k_{0}}{2\pi}\int_{\Gamma}\frac{\log f}{\tau}d\xi=:\beta_{f}.
        \label{4.6}
\end{equation}
gives some relevant information on the structure of the holomorphic $\Sigma$-factorization of $f$, by stating in (ii) that a special $\Sigma$-factorization of $f$ exists if and only if $\beta_{f}=0$ $\mod L$. The following theorem shows that indeed the rational middle factor $r$ in (\ref{4.1}) can be expressed as a power of exponent $N\leq 3$ of a rational function $r_{\nu}$ of the form (\ref{3.16}), where both $\nu$ and $N$ are determined by $\beta_{f}$. In particular, it is possible to establish conditions for $r$ to be a (non-constant) rational function of the simplest form (with two simple zeros and two simple poles), in the case where a special $\Sigma$-factorization does not exist.
\begin{thm}\label{theo.4.2}
Let $f\in \mathcal{G}C_{\mu}(\Gamma)$ be such that $\ind_{j}f=0$, for $j=1,2$, and let $\beta_{f}$ be defined by (\ref{4.6}) for some branch of the logarithm such that $\log f \in C_{\mu}(\Gamma)$. Then a holomorphic $\Sigma$-factorization for $f$ is given by
\begin{equation}
        f=f_{-}r_{\nu}^kf_{+}
\label{4.7}
\end{equation}
where, for $\tilde{\beta}_f \in \tilde{\mathcal{P}}$ such that $\beta_{f}=\tilde{\beta}_f \mod L$,
\begin{itemize}
  \item[(i)]
  \begin{itemize}
    \item[] $k=0$  if  $\tilde{\beta}_f=0$;
    \item[] $k=1$ if  $\tilde{\beta}_f \in \mathcal{P}_{1}\backslash\{0\}$  (see (\ref{3.13}));
    \item[] $k=2$ if  $\tilde{\beta}_f \in \tilde{\mathcal{P}}\backslash\mathcal{P}_{1},$ Re $\tilde{\beta}_f<2K$;
    \item[] $k=3$  if  Re $\tilde{\beta}_f=2K$;
  \end{itemize}
\item[(ii)] for $k\neq 0$, $r_{\nu}$ is given by (\ref{3.16})-(\ref{3.14}) with  $\beta=\tilde{\beta}_f/k$; for $k=0$, $r_{\nu}=1$ by convention;
  \item[(iii)] $f_{\pm}=\exp \tilde{P}^{\pm}_{\Gamma}(\Log (r^{-k}_{\nu}f))$.
\end{itemize}
\end{thm}
\begin{proof}
From Theorem \ref{theo4.1}, if $\beta_{f}=0 \mod L$ then $f$ admits a special $\Sigma$-factorization, with $f_{\pm}$ given by the equality in (iii) with $r_{\nu}^{-k}=1$. In all other cases,$\tilde{f}=r_{\nu}^{-k}f$ is such that $\ind_{j}\tilde{f}=0$ for $j=1,2$ and we see from Theorem \ref{theo3.5} that it admits a special $\Sigma$-factorization with $\tilde{f}_{\pm}=f_{\pm}$ given in (iii).
\end{proof}
In particular we have the following, which will later be used.
\begin{cor}\label{cor4.3}
With the same assumptions as in Theorem \ref{theo.4.2}, $f$ admits a holomorphic $\Sigma$-factorization $f=f_{-}\tau/[(\xi+i)(\xi-ik_{0})]f_{+}$
if and only $\beta_{f}=iK' \mod L$.
\end{cor}
Using the results of section 3 we can now extend the previous results to any $f \in \mathcal{G}C_{\mu}(\Gamma)$. In what follows we use the notation $n_{j}=\ind_{j}f, \quad j=1,2,$
and we define $\ind_{\Gamma}f=\ind_{1}f+\ind_{2}f$
for any $f \in \mathcal{G}C_{\mu}(\Gamma)$.
\begin{thm}
Every $\tilde{f}\in \mathcal{G}C_{\mu}(\Gamma)$  admits a holomorphic $\Sigma$-factorization.
\end{thm}
\begin{proof}
Let $\lambda_{\pm}(\xi)=\xi\pm i$ as in (\ref{3.20C}) and let
\begin{equation*}
        f=\tilde{f}(\frac{\lambda_{-}}{\lambda_{+}})^{-\tilde{k}}S^l(\alpha_{+}^{-1}(\alpha_{+})_*)^m
\label{4.8C}
\end{equation*}
(see Theorem \ref{theo3.4} for $S$ and (\ref{3.19}) for $\alpha_{+}$), where
\begin{equation}
        \tilde{k}=\frac{n_{1}+n_{2}}{2}, \  l=0, \ m=\frac{n_{1}-n_{2}}{2}, \text{ if } \ind_{\Gamma}f \text{ is even},
\label{4.8D}
\end{equation}
\begin{equation}
        \tilde{k}=\frac{n_{1}+n_{2}-1}{2}, \  l=1, \ m=\frac{n_{1}-n_{2}-1}{2}, \text{ if } \ind_{\Gamma}f \text{ is odd}.
\label{4.8E}
\end{equation}
From Theorem \ref{theo3.4} and Corollary \ref{cor3.7}, it follows that $\ind_{j}f=0$, so that $f$ admits a holomorphic $\Sigma$-factorization (\ref{4.7}). Therefore $\tilde{f}$ also admits a holomorphic $\Sigma$-factorization $\tilde{f}=\tilde{f}_{-}r\tilde{f}_{+}$
with
\begin{equation*}
        \tilde{f}_{-}=f_{-}, \quad r=r_{\nu}^k(\frac{\lambda_{-}}{\lambda_{+}})^{\tilde{k}}S^{-l}, \quad \tilde{f}_{+}=f_{+}(\alpha_{+}^{-1}(\alpha_{+})_*)^{-m}.
\label{4.8G}
\end{equation*}
\end{proof}
We remark that the proof of the last theorem provides formulas for the factors in a holomorphic $\Sigma$-factorization of any $\tilde{f} \in \mathcal{G}C_{\mu}(\Gamma)$. Moreover, since for any *-invariant $f \in \mathcal{G}C_{\mu}(\Gamma)$ (which can be identified with a function in $\mathcal{G}C_{\mu}(\dot{\mathbb{R}}))$ we have $\tilde{k}=\ind f$, $l=0$, $ m=0$, taking Theorem \ref{theo4.1} and (\ref{2.26A}) into account we conclude that the holomorphic $\Sigma$-factorization of $f$ coincides with its WH-factorization.
\begin{thm}\label{theo4.5}
Let $\tilde{f} \in \mathcal{G}C_{\mu}(\Gamma)$, with $\ind_{\Gamma}\tilde{f}=0$. Then $\tilde{f}$ admits a special $\Sigma$-factorization if and only if
\begin{equation}
        \beta_{\tilde{f}}=2nK \mod  L
\label{4.8H}\end{equation}
where $n=\ind_{1}\tilde{f}=-\ind_{2}\tilde{f}$. In this case, $\tilde{f}=\tilde{f}_{-}\tilde{f}_{+}$ with $\tilde{f}_{-}=f_{-}, \linebreak \tilde{f}_{+}=f_{+}(\alpha_{+}^{-1}(\alpha_{+})_*)^{-n}$
where $f_{\pm}$ are the factors of a special $\Sigma$-factorization of $f=\tilde{f}(\alpha_{+}^{-1}(\alpha_{+})_*)^{n}$, given by (\ref{4.4}).
\end{thm}
\begin{proof}
From Corollary \ref{cor3.7}, we have $\alpha_{+}^{-1}(\alpha_{+})_* \in \mathcal{G}C_{\mu}^+(\Gamma)$, so that $\tilde{f}$ admits a special $\Sigma$-factorization if and only if $f$ does. On the other hand, from  Corollary \ref{cor3.7}, we have $\ind_{1}f=\ind_{2}f=0$ and, from Theorem \ref{theo4.1}, $f$ admits a special $\Sigma$-factorization if and only if $\beta_{f}=0 \mod L$. The result now follows from (\ref{3.24}).
\end{proof}
\begin{rem} \label{rmk4.6}A generalization of Corollary \ref{cor4.3} can also be obtained following the same reasoning and with the same assumptions as in Theorem \ref{theo4.5}: $\tilde{f}$ admits a factorization of the form $\tilde{f}=\tilde{f}_{-}\tau/[(\xi+i)(\xi-ik_{0})]\tilde{f}_{+}$
(which will be important in the next section) if and only if $\beta_{\tilde{f}}=2nK+iK' \mod L$.
\end{rem}
The usefulness of holomorphic $\Sigma$-factorization in solving Riemann-\linebreak[4]Hilbert problems (relative to $\Gamma$) in $\Sigma$ has been illustrated in examples presented in (\cite{CSS,CSS1}). In particular it has been used to study and solve some boundary value problems which appear in connection with finite-dimensional integrable systems.
The examples presented there, however, make it clear that the difficulty in solving RH-problems, such as those mentioned above, considerably increases with the complexity of the rational function $r$ in (\ref{4.1}). This is particularly true if the RH problem on $\Sigma$ is studied as a step to obtain the explicit factorization of matrix functions, as will be done in section 6.

In this context, the simplest case corresponds naturally to the existence of a special $\Sigma$-factorization for $f$. Otherwise, the simplest case involves two simple zeros and two simple poles in $r$, the other cases involving two zeros and two poles of order $N>1$ (\cite{CSS}).

An alternative approach to $\Sigma$-factorization consists in looking for a factorization of the form (\ref{4.1}), allowing however the outer factors $f_{\pm}$ to have some known zeros or poles, which leads to a factorization involving a smaller number of zeros and poles in $\Sigma^{\pm}$ than in the case of holomorphic outer factors and rational middle factor. This alternative approach is presented in the next section.
\section{Meromorphic $\Sigma$-factorization}
By a \emph{meromorphic $\Sigma$-factorization} (relative to $\Gamma$) of $f \in C_{\mu}(\Gamma)$ we mean a representation of the form
\begin{equation}
        f=m_{f}^-rm_{f}^+
\label{5.1}
\end{equation}
where $(m_{f}^{\pm})^{\pm 1} \in \mathcal{M}(\Sigma^{\pm})$, $r \in \mathcal{R}(\Sigma)$.

It is clear that a necessary condition for existence of a representation (\ref{5.1}) is that $f \in \mathcal{G}C_{\mu}(\Gamma)$. It is also clear that any holomorphic $\Sigma$-factorization is of the above type. If the factors $m_{f}^{\pm}$ or their inverses are not holomorphic in $\Sigma^{\pm}$, respectively, we say that (\ref{5.1}) is a strictly meromorphic $\Sigma$-factorization.

Since, for the same $f \in \mathcal{G}C_{\mu}(\Gamma)$ admitting a meromorphic $\Sigma$-factorization, (infinitely) many such representations are possible, it is important to characterize more precisely the factors on the right-hand side of (\ref{5.1}).

Here we will consider meromorphic $\Sigma$-factorizations allowing the factors $m_{f}^{\pm}$ to be non-holomorphic in $\Sigma^{\pm}$ (respectively) but, in a certain sense, of the simplest non-holomorphic kind, admitting one simple pole and no zeros in $\Sigma^{\pm}$, respectively. In what follows we will thus consider that the meromorphic $\Sigma$-factorization is of the form
\begin{equation}
f=f_{-}^M\alpha_{-}^l(\frac{\lambda_{-}}{\lambda_{+}})^{\tilde{k}}r_{\nu}^kf_{+}^M
\label{5.2}\end{equation}
(so that $m_{f}^-=f_{-}^M\alpha_{-}^l$,  $r=(\lambda_{-}/\lambda_{+})^{\tilde{k}}r_{\nu}^k$, $m_{f}^+=f_{+}^M$ when comparing with (\ref{5.1})) where $k, \tilde{k} \in \mathbb{Z}$,\quad $f_{-}^M \in \alpha_{-}C_{\mu}^-(\Gamma), \ (f_{-}^M)^{-1} \in C_{\mu}^-(\Gamma)$, \quad $f_{+}^M \in \alpha_{+}C_{\mu}^+(\Gamma), \linebreak (f_{+}^M)^{-1} \in C_{\mu}^+(\Gamma)$,
$r_{\nu}$ is a rational function defined, for a given $\nu \in \mathbb{C}$, by (\ref{3.16}) and $l \in \{0,1\}$, $l=\ind_{\Gamma}f(\mod 2)$.

If $l=k=\tilde{k}=0$ and $f_{\pm}^M=F_{\pm}\alpha_{\pm}$ with $F_{\pm}\in \mathcal{G}C_{\mu}^{\pm}(\Gamma)$, we have
\begin{equation}
f=(F_+\alpha_+)(\alpha_-F_-);
\label{5.4A}\end{equation}
we say then that (\ref{5.4A}) is an \emph{$M$-special $\Sigma$-factorization}.

It is easy to see that the factors $f_{\pm}^M$ in this case can differ only by a non-zero constant factor, analogously to what happened in the case of existence of a special $\Sigma$-factorization.

As in the previous section, we start by studying the case where \linebreak $\log f \in C_{\mu}(\Gamma)$.
In what follows, let $\beta_{f}$ be defined by (\ref{4.6}) and let $\tilde{\beta}_f \in \tilde{\mathcal{P}}$ be such that $\beta_{f}=\tilde{\beta}_f \mod L$.
\begin{thm}\label{theo5.1}
Let $f \in \mathcal{G}C_{\mu}(\Gamma)$, with $\ind_{j} f=0$ for $j=1,2$. Then $f$ admits a meromorphic $\Sigma$-factorization (\ref{5.2}) where $l=\tilde{k}=0$,
\begin{equation}\label{5.5AA}
    f=f_{-}^{M}r_{\nu}^kf_{+}^{M}
\end{equation}
 where
\begin{itemize}
  \item[(i)]  if $\tilde{\beta}_f \in \mathcal{P}_{1}\cup \partial\mathcal{P}_{1}$, then $k$ and $r_{\nu}^k$ are defined as in Theorem \ref{theo.4.2}, while $f_{\pm}^M$ coincide with $f_{\pm}= \exp \tilde{P}^{\pm}_{\Gamma}(\Log r_{\nu}^{-k}f)$, respectively;
  \item[(ii)] if
  $\tilde{\beta}_f \in \tilde{\mathcal{P}}\backslash(\mathcal{P}_{1}\cup \partial\mathcal{P}_{1})$, then $r_{\nu}^k$ coincides with the rational middle factor in the holomorphic $\Sigma$-factorization $F=F_{-}r_{\nu}^kF_{+}$ defined for $F=f\alpha_{-}^{-1}\alpha_{+}^{-1}$ in Theorem \ref{theo.4.2}, and we have $f_{\pm}^M=F_{\pm}\alpha_{\pm}$.
\end{itemize}
\end{thm}
\begin{proof}
(i) is straightforward; (ii)  is a  consequence of Theorems \ref{theo3.6} and \ref{lemma3.2}, which imply that if $\beta_F=\tilde{\beta}_F \mod L$ and $\tilde{\beta}_F \in \tilde{\mathcal{P}}$, then actually $\tilde{\beta}_F \in \mathcal{P}_{1}$, and the result follows from Theorem \ref{theo.4.2}.
\end{proof}
Consequently, taking also Theorem \ref{theo4.1} (ii) into account, we have:

\begin{cor}\label{cor5.2}
With the same assumptions as in Theorem \ref{theo5.1}, $f$ admits an $M$-special $\Sigma$-factorization if and only if $\tilde{\beta}_f=2K$. In this case, \linebreak$f=(F_{-}\alpha_{-})(\alpha_{+}F_{+})$,
\begin{equation}
        F_{\pm}=\exp \tilde{P}_{\Gamma}^{\pm}(\Log F), \quad F=f\alpha_{-}^{-1}\alpha_{+}^{-1}.
\label{5.6}
\end{equation}
\end{cor}
\begin{cor}
With the same assumptions as in Theorem \ref{theo5.1}, if either  $\tilde{\beta}_f \in
\left\{s+it: s\in ]K,2K], t \in ]-iK',iK']\right\}\backslash\{2K\}$
or   $\tilde{\beta}_f \in \left\{s+it: s\in ]-2K,-K[, \right.$ \linebreak $\left. t \in ]-iK',iK']\right\}
$, then $f$ admits a meromorphic $\Sigma$-factorization
\begin{equation}
        f=(F_{-}\alpha_{-})r_{\nu}(\alpha_{+}F_{+}),
\label{5.7}
\end{equation}
with $F_{\pm} \in \mathcal{G}C_{\mu}^{\pm}(\Gamma)$ and $r_{\nu}$ of the form (\ref{3.16}), where $F=f\alpha_{-}^{-1}\alpha_{+}^{-1}r_{\nu}^{-1}$.
\end{cor}
We remark that, in the cases considered in the two previous corollaries, $f$ also admits a holomorphic $\Sigma$-factorization that can be obtained according to Theorem \ref{theo.4.2}. In particular for the case where $\beta_{f}=2K \mod L$, a zero and a pole of order 3 have to be considered in the rational middle factor, as regards either $\Sigma^+$ or $\Sigma^-$. In contrast with this, in (\ref{5.6}) only one simple pole is involved whether we consider the factors which are meromorphic in $\Sigma^+$ or in $\Sigma^-$.

As a consequence of Theorem \ref{theo3.6} and Corollary \ref{cor3.7}, the case where $\ind_{j}f\neq 0$ for some $j=1,2$ can be reduced, as in the previous section, to that where $\ind_{1}f=\ind_{2}f=0$, using the meromorphic factors $\alpha_{+}^{\pm 1}$ and $(\alpha_{+})_*$. In what follows we define, for $\tilde{f} \in \mathcal{G}C_{\mu}(\Gamma)$,
\begin{equation}
        f=\tilde{f}(\frac{\lambda_{-}}{\lambda_{+}})^{-\tilde{k}}\alpha_{-}^{-l}(\alpha_{+}^{-1}(\alpha_{+})_*)^{m}
\label{5.5.A}
\end{equation}
with $\tilde{k}, l$ as in (\ref{4.8D}) and (\ref{4.8E}) and  $m=(n_{1}-n_{2})/2 \text{ if } \ind_{\Gamma}f \text{ is even},$  $m=(n_{1}-n_{2}+1)/2 \text{ if } \ind_{\Gamma}f \text{ is odd}$.

 It follows from Theorem \ref{theo3.6} and Corollary \ref{cor3.7} that $\ind_{j}f=0$ for $j=1,2$, so that $f$ admits a meromorphic $\Sigma$-factorization of the form (\ref{5.2}), according to Theorem \ref{theo5.1}. With $f_{\pm}^M$ and $r_{\nu}^k$ defined as in Theorem \ref{theo5.1} and $\tilde{k}, l, m$ defined as in the previous paragraph, it is easy to see that the following holds.
\begin{thm}\label{theo5.4}
Every $\tilde{f}\in \mathcal{G}C_{\mu}(\Gamma)$ admits a meromorphic $\Sigma$-factorization
\begin{equation*}
        \tilde{f}=\tilde{f}^{M}_{-}\alpha_{-}^l(\frac{\lambda_{-}}{\lambda_{+}})^{\tilde{k}}r_{\nu}^k\tilde{f}^{M}_{+}
\end{equation*}
with $\tilde{f}^{M}_{-}=f^M_{-}$, $\tilde{f}^{M}_{+}=f^M_{+}(\alpha_{+}^{-1}(\alpha_{+})_*)^{-m}$.
\end{thm}
\begin{thm}\label{theo5.5}
Let $\tilde{f} \in \mathcal{G}C_{\mu}(\Gamma)$, $\ind_{\Gamma}\tilde{f}=0$. Then $\tilde{f}$ admits an $M$-special $\Sigma$-factorization if and only if $\beta_{\tilde{f}}=2(n+1)K \mod L$
where $n=\ind_1 \tilde{f}=-\ind_2 \tilde{f}$.
\end{thm}
We conclude this section with the following remark which allows a better understanding of these results and may deserve further investigation in the future.

It is clear that to each $\tilde{f}\in \mathcal{G}C_{\mu}(\Gamma)$ we can associate by (\ref{5.5.A}) a unique $f \in \mathcal{G}C_{\mu}(\Gamma)$ such that $\ind_jf=0, \ j=1,2$ and reduce the study of the $\Sigma$-factorization of $\tilde{f}$ to that of $f$. On the other hand, to each such $f$ we can associate a unique $\tilde{\beta}_f \in \tilde{\mathcal{P}}$ such that $\beta_f=\tilde{\beta}_f \mod L$ and, thus, a unique point $P_f$ in $\Sigma$, $P_f=\sigma(\tilde{\beta}_f)=\sigma(\beta_f)$.

Now, according to the results of sections 4 and 5, we conclude that $P_f$ determines the form of the $\Sigma$-factorization of $f$. For instance, considering meromorphic $\Sigma$-factorizations of the form (\ref{5.5AA}), if $P_{f_1} \in \Sigma_1$ and $P_{f_2} \in \Sigma_2$ have the same projection in $\mathbb{C}$, then the factorizations of $f_1$ and $f_2$ differ by a factor $\alpha_-\alpha_+$.
In particular we see from Theorem \ref{theo4.1} and Corollary \ref{cor5.2} that, for $f$ such that $\ind_jf=0, \ j=1,2,$ the existence of a special $\Sigma$-factorization corresponds to having $P_f=\mathbf{0_1}$, while the existence of an M-special $\Sigma$-factorization corresponds to $P_f=\mathbf{0_2}$.
\section{$\Sigma$-factorization and Riemann-Hilbert problems}
In this section we apply the results of the previous sections to the study of some vectorial  RH problems that can be equivalently formulated as scalar RH problems relative to a contour on a Riemann surface.

Let $G \in \mathcal{G}(C_{\mu}(\dot{\mathbb{R}}))^{2\times 2}$ be a matrix function satisfying a relation
\begin{equation}
        G^{T}QG=\det G. Q,
\label{6.1}
\end{equation}
where $Q$ is a symmetric rational matrix whose entries do not have poles on $\dot{\mathbb{R}}$ and possesses an inverse of the same type, and $\det G $ admits a bounded factorization
\begin{equation}
        \det G=\gamma_{-}(\frac{\lambda_{-}}{\lambda_{+}})^{m}\gamma_{+}
\label{6.2}
\end{equation}
with $\gamma_{\pm} \in \mathcal{G}C_{\mu}^{\pm}(\dot{\mathbb{R}})$, $m \in \mathbb{Z}$ and $\lambda_{\pm}$ defined by (\ref{3.20C}).
We denote by $C(Q)$ the class of matrix functions $G\in \mathcal{G}(C_\mu(\dot{\mathbb{R}}))^{2\times 2}$ satisfying (\ref{6.1}) (\cite{CSS, CSCarpentier}).

RH problems of the form
\begin{equation}
         G\varphi_{+}=\varphi_{-}+\eta
\label{6.4}
\end{equation}
which consist in, given a matrix function $G$ satisfying (\ref{6.1}) and a vector function $\eta$, finding $\varphi_{\pm}$ in appropriate spaces of analytic vector functions, appear in connection with many problems in Physics, Engineering and Mathematics (see for instance \cite{CSS1, Daniele, Khrapkov, Kup, MeSp1, MeSp2}). Let us assume that, in (\ref{6.4}), $G \in C(Q)$, $\eta \in (L_p(\mathbb{R}))^2$ with $1<p<\infty$ and $\varphi_{\pm}$ are sought in the Hardy spaces $(H_p^{\pm})^2$, with $H_p^{\pm}:=H_p(\mathbb{C}^{\pm})$ (\cite{Duren, Koosis}). In this case, (\ref{6.4}) is equivalent to the equation $T_G\varphi_+=\eta_+$
where $T_G:(H^+_p)^2\rightarrow (H^+_p)^2$, $T_G\varphi_+=P^+(G\varphi_+)$, is the Toeplitz operator in $(H^+_p)^{2}$  with symbol $G$  and $\eta_+=P^+\eta$.

Since $G\in \mathcal{G}(C_\mu(\dot{\mathbb{R}}))^{2\times 2}$, it admits a representation as a product \cite{MP, CG, LS}
\begin{equation}\label{6.6.1}
G=G_-DG_+
\end{equation}
with
\begin{equation}\label{6.7.1}
G_{\pm}\in \mathcal{G}(C_{\mu}^{\pm})^{2\times 2}, \quad  D= \diag ((\lambda_-/\lambda_+)^{k_1}, \ (\lambda_-/\lambda_+)^{k_2})
\end{equation}
where $k_1, k_2 \in \mathbb{Z}$ are uniquely defined, up to their order, and are called the \emph{partial indices} of $G$. We have moreover $k_1+k_2=m$, and if we assume (for reasons that will be explained later in this section) that $m \in \{0,1\}$, we can write
\begin{equation}\label{6.8.1}
k_1=-k \text{ for some } k\geq 0, \ k_2=k+m.
\end{equation}
A complete solvability picture of (\ref{6.4}) can be obtained from the factorization (\ref{6.6.1}), as well as many properties of the Toeplitz operator $T_G$ (\cite{MP, CG, LS}). Namely, $T_G$ is invertible if and only if $k_1=k_2=0$ (the factorization \linebreak$G=G_-G_+$ being then called \emph{canonical}) and $(T_G)^{-1}$ can be expressed in terms of the factors $G_{\pm}$: $(T_G)^{-1}=G_+^{-1}P^+G_-^{-1}I_+$, where $I_+$ denotes the identity operator in $(H_p^+)^2$.

The partial indices of $G$ are not, in general, known a priori. They can however be determined by solving the homogeneous ($\eta=0$) RH problem
\begin{equation}
         G\varphi_{+}=\varphi_{-}, \quad \varphi_{\pm}\in (H_p^{\pm})^2
\label{6.9.1}
\end{equation}
which is equivalent to the problem of characterizing $\ker T_G$. In fact, assuming $m \in \{0,1\}$, the integer $k$ in (\ref{6.8.1}) is equal to the dimension of the space of solutions to (\ref{6.9.1}), and to the dimension of $\ker T_G$. It is not difficult to see on the other hand that, due to (\ref{6.6.1}), (\ref{6.7.1}) and (\ref{6.8.1}), $(\varphi_+,\varphi_-)$ is a solution to (\ref{6.9.1}) if and only if $\phi_{\pm}=\lambda_{\pm}\varphi_{\pm}$ satisfy
\begin{equation}
         G\phi_{+}=\phi_{-}, \quad \phi_{+}\in (C_{\mu}^+(\dot{\mathbb{R}}))^2, \ \phi_{-}\in (C_{\mu 0}^-(\dot{\mathbb{R}}))^2
\label{6.10.1}
\end{equation}
where $C_{\mu 0}^-(\dot{\mathbb{R}})=(\lambda_{+}/\lambda_-)C_{\mu}^-(\dot{\mathbb{R}})$. We will thus study here the vectorial RH problem (\ref{6.10.1}), assuming that $Q$ and $G$ take the normal forms associated with $C(Q)$ (\cite{CMalheiro}),
\begin{equation}
         Q=\left[
             \begin{array}{cc}
               -q & 0 \\
               0 & 1 \\
             \end{array}
           \right], \quad G=\left[
             \begin{array}{cc}
               \alpha & \delta \\
               q\delta & \alpha \\
             \end{array}
           \right]
\label{6.6}
\end{equation}
with $\alpha$, $\delta \in C_{\mu}(\dot{\mathbb{R}})$ and
\begin{equation}
         q=-\det Q=\frac{\mathfrak{p}_{1}}{\mathfrak{p}_{2}}
\label{6.7}
\end{equation} where $\mathfrak{p}_{1}(\xi)=(\xi+i)(\xi+ik_{0}), \quad \mathfrak{p}_{2}(\xi)=(\xi-i)(\xi-ik_{0})$.
The rational function $q$ in (\ref{6.7}) is related to the polynomial $\mathfrak{p}$ defined by the right-hand side of (\ref{2.2}) by $q=\mathfrak{p}\mathfrak{p_{2}^{-2}}=\mathfrak{p_{1}^2}\mathfrak{p}^{-1}$
and we say that $\Sigma$, defined as in section 2, is the \emph{Riemann surface associated to $C(Q)$}.

Let now $T_{\Sigma}:(C_{\mu}(\dot{\mathbb{R}}))^2\rightarrow C_{\mu}(\Gamma)$ be the linear transformation defined by $T_{\Sigma}(\varphi_{1}, \varphi_{2})_{|\Gamma_{j}}=\varphi_{j}, \quad j=1,2,$
for which it is easy to see that the following holds (\cite{CSS}).
\begin{prop}\label{prop6.1}
\begin{itemize}
  \item[(i)] $T_{\Sigma}$ maps $(\varphi_{1}+\rho\varphi_{2}, \varphi_{1}-\rho\varphi_{2})$ into $\varphi_{1}+\tau\varphi_{2}$.
  \item[(ii)] $T_{\Sigma}$ is invertible with inverse $T_{\Sigma}^{-1}$ given by
  \begin{equation*}
  T_{\Sigma}^{-1} : C_{\mu}(\Gamma)  \longrightarrow  C_{\mu}(\dot{\mathbb{R}})^2, \quad
     T_{\Sigma}^{-1}( \phi) = (\phi_{|\Gamma_{1}}, \phi_{|\Gamma_{2}} ).
  \label{3.2.A}
\end{equation*}
\end{itemize}
\end{prop}
By diagonalizing $G$ and taking $\phi_{\pm}=(\phi_{1\pm},\phi_{2\pm})$ we can rewrite (\ref{6.10.1}) in the equivalent form
\begin{equation}
           \left\{\begin{array}{c}
                    \hspace{-0.3cm}g_{1}(\phi_{1+}+\frac{\rho}{\mathfrak{p_{1}}}\phi_{2+})= \frac{\rho}{\mathfrak{p_{1}}}(\phi_{2-}+\frac{\rho}{\mathfrak{p_{2}}}\phi_{1-})\\
                    g_{2}(\phi_{1+}-\frac{\rho}{\mathfrak{p_{1}}}\phi_{2+})= -\frac{\rho}{\mathfrak{p_{1}}}(\phi_{2-}-\frac{\rho}{\mathfrak{p_{2}}}\phi_{1-})
                  \end{array}\right.
\label{6.10}
\end{equation}
where  $g_{1}$ and $g_{2}$ are the eigenvalues $g_{1}=\alpha+\rho\delta\mathfrak{p}_{2}^{-1}, \quad g_{2}=\alpha-\rho\delta\mathfrak{p}_{2}^{-1}$
for $\rho=\sqrt{\mathfrak{p}}$ defined as in Section 2. It follows from Proposition \ref{prop6.1} in \cite{CSS} (see also \cite{MEISTER}) that (\ref{6.10}) is equivalent  to the scalar RH problem relative to $\Gamma$ in $\Sigma$
\begin{equation}
           g\psi_{+}=\frac{\tau}{\mathfrak{p}_{1}}\psi_{-}\text{ with }, \quad \psi_+ \in C_{\mu}^+(\Gamma), \ \psi_-\in C_{\mu 0}^-(\Gamma),
\label{6.12}
\end{equation} where $C_{\mu 0}^-(\Gamma)=(\lambda_+/\lambda_-)C_{\mu}^-(\Gamma)$ and
\begin{equation}g=T_{\Sigma}(g_{1}, g_{2})=\alpha+\frac{\tau}{\mathfrak{p}_{2}}\delta\label{6.12A}\end{equation}
 will be called the \emph{$\Sigma$-symbol of $G$}.

Multiplying $G$ by a rational factor $(\lambda_{-}/\lambda_{+})^{-m/2}$ if $m$ is even, $(\lambda_{-}/\lambda_{+})^{-(m-1)/2}$ if $m$ is odd, we obtain a matrix which also satisfies (\ref{6.1}) but whose determinant admits a bounded factorization of the form (\ref{6.2}) with $m=0$ or $m=1$. Thus we assume in the results that follow that $m \in \{0,1\}$ in (\ref{6.2}) and we will consider separately the cases where $m=0$ and $m=1$.
\begin{thm}\label{theo6.1}
Let $G \in C(Q)$ be such that $m=0$ in (\ref{6.2}) and let $g$ be its $\Sigma$-symbol. Then the following propositions are equivalent:
\begin{itemize}
\item[(i)]The RH problem (\ref{6.10.1}) admits non-zero solutions.
  \item[(ii)] The RH problem (\ref{6.12}) admits non-zero solutions.
  \item[(iii)] The $\Sigma$-symbol of $G$ admits a holomorphic $\Sigma$-factorization
  \begin{equation}
           g=g_{-}r_{0}g_{+} \quad \text{with } r_{0}=\frac{\tau}{(\xi+i)(\xi-ik_{0})}
\label{6.13}
\end{equation}
\item[(iv)]$\beta_{g}=\frac{k_{0}}{2\pi}\int_{\Gamma}\frac{\log g}{\tau}d\xi=2nK+iK' \mod L$, where $n=\ind_{1}g=-\ind_{2}g$.
\end{itemize}
\end{thm}
\begin{proof} (i)$\Leftrightarrow$(ii) since $(\phi_{+}, \phi_{-})$ is a solution to (\ref{6.10.1}) if and only if \linebreak$\psi_{+}=\phi_{1+}+(\tau/\mathfrak{p_{1}})\phi_{2+}$, $\psi_{-}=\phi_{2-}+(\tau/\mathfrak{p_{2}})\phi_{1-}$ satisfy (\ref{6.12}).

(ii)$\Rightarrow$(iii) If (ii) is true and $(\psi_+,\psi_-)$ is a non-zero solution to (\ref{6.12}), then we have
\begin{equation}
           g\psi_{+}=\frac{\tau}{(\xi-i)(\xi+ik_{0})}\widetilde{\psi}_{-}
\label{6.14}
\end{equation} with $\widetilde{\psi}_{-}=(\xi-i)/(\xi+i)\psi_{-} \in C_{\mu}^-(\Gamma)$, since $\psi_{-}\in C_{\mu 0}^-(\Gamma)$. Applying the involution * to both sides of (\ref{6.14}) and multiplying, we obtain
$$gg_*\psi_{+}(\psi_{+})_*=-\frac{(\xi+i)(\xi-ik_{0})}{(\xi-i)(\xi+ik_{0})}\widetilde{\psi}_{-}(\widetilde{\psi}_{-})_*.$$
Taking into account that $gg_{*}=g_{1}g_{2}=\det
G=\gamma_{-}\gamma_{+}$, we have thus
\begin{equation}
      \gamma_{+}\frac{\xi+ik_{0}}{\xi+i}\psi_{+}(\psi_{+})_*=-\gamma_{-}^{-1}\frac{\xi-ik_{0}}{\xi-i}\widetilde{\psi}_{-}(\widetilde{\psi}_{-})_*
\label{6.15}
\end{equation}
and, since the left and the right hand sides of (\ref{6.15}) can be identified with functions in $C_{\mu}^+(\dot{\mathbb{R}})$ and $C_{\mu}^-(\dot{\mathbb{R}})$, respectively, both sides must be equal to  $c \in \mathbb{C}\backslash\{0\}$. Thus $\psi_{+}$ and $\widetilde{\psi}_{-}$ are bounded away from zero in $\Sigma^{+}$ and $\Sigma^{-}$, respectively, and therefore their inverses are also in $C_{\mu}^+(\Gamma)$ and $C_{\mu}^-(\Gamma)$, respectively. From (\ref{6.14}) we obtain then
$$g=\widetilde{\psi}_{-}\frac{\tau}{(\xi-i)(\xi+ik_{0})}\psi_{+}^{-1}=\left(\widetilde{\psi}_{-}
\frac{\xi-ik_{0}}{\xi-i}\right)\frac{\tau}{(\xi+i)(\xi-ik_{0})}\left(\frac{\xi+i}{\xi+ik_{0}}\psi_{+}^{-1}\right)$$
and we can take $g_{-}=
[(\xi-ik_{0})/(\xi-i)]\widetilde{\psi}_{-}$, $g_{+}=[(\xi+i)/(\xi+ik_{0})]\psi_{+}^{-1}$.

(iii)$\Rightarrow$(ii) It is enough to take
$\psi_{+}=[(\xi+i)/(\xi+ik_{0})]g_{+}^{-1}$,\linebreak
$\psi_{-}=[(\xi+i)/(\xi-ik_{0})]g_{-}$.

(iii)$\Leftrightarrow$(iv) Since $m=0$, we must
have $\ind_{1}g=\ind g_{1}=n$ and $\ind_{2}g=\ind g_{2}=-n$ for some $n \in \mathbb{Z}$
and the equivalence follows as in Remark \ref{rmk4.6}.
\end{proof}
\begin{thm}\label{theo6.2}
 Let the assumptions of Theorem \ref{theo6.1} hold and let the $\Sigma$-symbol $g$ admit a factorization (\ref{6.13}). Then the space of solutions $(\psi_{+},\psi_{-})$ of (\ref{6.12}) is generated by $([(\xi+i)/(\xi+\nolinebreak ik_{0})]g_{+}^{-1}, [(\xi+\nolinebreak i)/(\xi-\nolinebreak ik_{0})]g_{-})$ and the space of solutions $(\phi_{+}, \phi_{-})$ to the RH problem (\ref{6.10.1}) is generated by $(\Phi_{+},\Phi_{-})$ with
\begin{equation}\Phi_{+}=\left(\frac{\xi+i}{\xi+ik_{0}}(g_{+}^{-1})_{\mathcal{E}}, \ (\xi+i)^2(g_{+}^{-1})_{\mathcal{O}}\right),\label{6.16A}\end{equation}
\begin{equation}\Phi_{-}=\left((\xi^2+1)(g_{-})_{\mathcal{O}}, \ \frac{\xi+i}{\xi-ik_{0}}(g_{-})_{\mathcal{E}}\right).
\label{6.16B}\end{equation}
\end{thm}
\begin{proof}
From (\ref{6.13}) we have
$$g\psi_{+}=\frac{\tau}{\mathfrak{p}_{1}}\psi_{-}\Leftrightarrow g_{+}\psi_{+}=g_{-}^{-1}\frac{\xi-ik_{0}}{\xi+ik_{0}}\psi_{-}.$$
Both sides of the latter equality must be equal to a rational function with (at most) a double pole at the branch point $-ik_{0}$ and a double zero at the branch point $-i$ (due to $\psi_{-}\in C_{\mu 0}^-(\Gamma)$). Thus
$$g_{+}\psi_{+}=g_{-}^{-1}\frac{\xi-ik_{0}}{\xi+ik_{0}}\psi_{-}=c \ \frac{\xi+i}{\xi+ik_{0}}, \quad c \in \mathbb{C}$$
and therefore
\begin{equation}\hspace{-0.5cm}\psi_{+}=c \ \frac{\xi+i}{\xi+ik_{0}}g_{+}^{-1}, \quad  \psi_{-}=c \ \frac{\xi+i}{\xi-ik_{0}}g_{-}, \quad c \in \mathbb{C}
\label{6.17}\end{equation}
give all the solutions to (\ref{6.12}).

The solutions to the RH problem (\ref{6.10.1}) can be
obtained from (\ref{6.17}) using the equivalence with (\ref{6.12})
(\cite{CSS}) which implies that
$\phi_{\pm}=(\phi_{1\pm},\phi_{2\pm})$ satisfy
(\ref{6.10.1}) if and only if
$$(\phi_{1+}+\frac{\rho}{\mathfrak{p}_{1}}\phi_{2+}, \ \phi_{1+}-\frac{\rho}{\mathfrak{p}_{1}}\phi_{2+})=T_{\Sigma}^{-1}\psi_{+}$$

$$=c \ \frac{\xi+i}{\xi+ik_{0}}\left((g_{+}^{-1})_{\mathcal{E}}+\rho(g_{+}^{-1})_{\mathcal{O}}, \  (g_{+}^{-1})_{\mathcal{E}}-\rho(g_{+}^{-1})_{\mathcal{O}}\right),$$

$$(\phi_{2-}+\frac{\rho}{\mathfrak{p}_{2}}\phi_{1-}, \ \phi_{2-}-\frac{\rho}{\mathfrak{p}_{2}}\phi_{1-})=T_{\Sigma}^{-1}\psi_{-}$$

$$\hspace{-0.3cm}=c \ \frac{\xi+i}{\xi-ik_{0}}\left((g_{-})_{\mathcal{E}}+\rho(g_{-})_{\mathcal{O}}, \  (g_{-})_{\mathcal{E}}-\rho(g_{-})_{\mathcal{O}}\right).$$
Thus we obtain $\Phi_{+}, \Phi_{-}$ given by (\ref{6.16A}), (\ref{6.16B}).
\end{proof}
\begin{rem}\label{rmk6.4}
The RH problem (\ref{6.12}) can also be studied, with the same assumptions as in Theorem \ref{theo6.2}, in a different setting, looking for solutions $\psi^{\pm}$ in $C_{\mu}^{\pm}(\Gamma)$ (see \cite{CSS}). It is easy to see, following the same reasoning as in the previous proof, that in that case the space of solutions is isomorphic to the space $L(-D)$ of meromorphic functions with poles bounded by the divisor $-D$ (\cite{Miranda}), where $D=$div$\left[(\xi-ik_{0})/(\xi+ik_{0})\right]_{|\Sigma^{-}}$, and thus its dimension is 2.
\end{rem}
As an immediate consequence of Theorems \ref{theo6.1} and \ref{theo6.2}, we conclude that $\ker T_G=\{0\}$ unless condition (iv) in Theorem \ref{theo6.1} is satisfied, in which case $\dim \ker T_G=1$ and $\ker T_G$ is generated by $\lambda_+^{-1}\Phi_+$ with $\Phi_+$ defined by (\ref{6.16A}). Therefore we can establish necessary and sufficient conditions for existence of a canonical factorization for $G$ (and invertibility of $T_G$), and determine the partial indices in the non-canonical case.
\begin{cor}\label{cor6.3}
With the same assumptions as in Theorem \ref{theo6.1}, $G$ admits a canonical bounded factorization unless $g$ admits a factorization (\ref{6.13}); in that case $G$ admits a non-canonical bounded factorization with partial indices $\pm 1$.
\end{cor}
Now we use Theorem \ref{theo6.1} to study the same problems when $m=1$ in (\ref{6.2}).
\begin{thm}\label{theo6.5}
Let $G \in C(Q)$ be such that $m=1$ in (\ref{6.2})  and let $g$ be its $\Sigma$-symbol. Then the RH problems (\ref{6.10.1}) and (\ref{6.12}) do not admit non-zero solutions and $G$ admits a non-canonical bounded factorization with partial indices $0$ and $1$.
\end{thm}
\begin{proof}
Let $\ind_{1}g=\ind g_{1}=n$, $\ind_{2}g=\ind g_{2}=-n+1$, with $n \in \mathbb{Z}$. As in the proof of Theorem \ref{theo6.2}, we use the equivalence between (\ref{6.10.1}) and (\ref{6.12}). Assume that $\psi_{\pm} \neq 0$ satisfy (\ref{6.12}).
Then, defining $\tilde{g}=(\alpha_{+}^{-1}(\alpha_{+})_*)^n\alpha_{+}\ g$,
we see from Theorem \ref{theo3.6} and
Corollary \ref{cor3.7} that $\ind_{j}\tilde{g}=0$ for $j=1,2$ and
the equation (\ref{6.12}) is equivalent to
\begin{equation}
\tilde{g}\eta_{+}=\frac{\tau}{\mathfrak{p_{1}}}\psi_{-}
\label{6.19}\end{equation}
where $\eta_{+}=(\alpha_{+}^{-1}(\alpha_{+})_*)^{-n}\alpha_{+}^{-1}\psi_{+} \in C_{\mu}^+(\Gamma)$ and $\eta_{+}$ has a zero at the branch point $i$ due to the factor $\alpha_{+}^{-1}$. Since $\eta_{+}, \psi_{-}\neq 0$, it follows from Theorem \ref{theo6.1} that $\tilde{g}$ admits a factorization
$\tilde{g}=\tilde{g}_{-}\tau/[(\xi+i)(\xi-ik_{0})]\tilde{g}_{+}$
so that, from (\ref{6.19}), we have
$\left[(\xi+ik_{0})/(\xi-ik_{0})\right]\tilde{g}_{+}\eta_{+}=\tilde{g}_{-}^{-1}\psi_{-}=q_{0}$
where $q_{0}\in \mathcal{R}(\Sigma)$ must have a double zero at the
branch point $-i$ (due to $\psi_{-} \in C_{\mu 0}^-(\Gamma)$), as
well as a zero at the branch point $i$ (due to the factor $\alpha
_{+}^{-1}$ in $\eta_{+}$) and, at most, a double pole at the branch point $ik_{0}$. Thus $q_{0}=0$, which implies that
$\psi_{\pm}=0$, against our assumption. Therefore (\ref{6.12}) has
only the trivial solution $\psi_{\pm}=0$. We conclude moreover that
(\ref{6.10.1}) admits also only the trivial solution
$\phi_{\pm}=0$ and therefore the partial indices of $G$ must be
non-negative (\cite{MP,CG,LS}). Since the total index of $G$ is
$1=\ind(\det G)=m$, it follows that the partial indices in a bounded
factorization of $G$ are $0, 1$.
\end{proof}
\begin{cor}
Let the assumptions of Theorem \ref{theo6.5} hold; then the Toeplitz operator $T_G$ in $(H_p^+)^2$ is injective, for all $p \in ]1, +\infty[$.
\end{cor}
Explicit formulas for a  WH factorization of $G$ can also be obtained
from a $\Sigma$-factorization of its $\Sigma$-symbol. This not only
illustrates the usefulness of the results in the previous sections,
but moreover shows the importance of determining
$\Sigma$-factorizations with factors of a particular type, like
$r_{\nu}$, $\alpha_{\pm}$. Indeed, as we show next, these factors
are $\Sigma$-symbols of certain important elements of $C(Q)$.
\begin{defn}
Let $\mathcal{I}:C(Q)\rightarrow \mathcal{G}C_{\mu}(\Gamma)$ be defined by $\mathcal{I}(G)=g$
where $g$ is the $\Sigma$-symbol of $G$ (cf. (\ref{6.12A})).
\end{defn}
It is easy to see that $C(Q)$ is a multiplicative group (\cite{CMalheiro}) and $\mathcal{I}$ is a group isomorphism. For $g \in \mathcal{G}C_{\mu}(\Gamma)$,
\begin{equation}
\mathcal{I}^{-1}(g)=\left[
                      \begin{array}{cc}
                        g_{\mathcal{E}} & \mathfrak{p}_{2}g_{\mathcal{O}} \\
                        \mathfrak{p}_{1}g_{\mathcal{O}} & g_{\mathcal{E}} \\
                      \end{array}
                    \right], \quad \mathcal{I}^{-1}(g_*)=\left[
                      \begin{array}{cc}
                        g_{\mathcal{E}} & -\mathfrak{p}_{2}g_{\mathcal{O}} \\
                        -\mathfrak{p}_{1}g_{\mathcal{O}} & g_{\mathcal{E}} \\
                      \end{array}
                    \right]=(\mathcal{I}^{-1}(g))^{\ast}
\label{6.21}\end{equation}where by $M^{\ast}$ we denote the adjugate (algebraic conjugate) of a matrix $M$.
Moreover, we have the following property:
\begin{prop} \label{prop6.7A}The image of $C(Q)\cap \mathcal{G}(C_{\mu}^{\pm}(\dot{\mathbb{R}}))^{2 \times 2}$ by $\mathcal{I}$ is $\mathcal{G}C_{\mu}^{\pm}(\Gamma)$.
\end{prop}
\begin{proof}
If $G \in C(Q)\cap (C_{\mu}^{+}(\dot{\mathbb{R}}))^{2 \times 2}$, then from (\ref{6.6}) we see that we must have $\alpha, \delta \in C_\mu^+(\dot{\mathbb{R}})$, $\delta (i)=\delta (ik_0)=0$ so that $g=\mathcal{I}(G)=\alpha+(\tau/p_2)\delta \in C_\mu^+ (\Gamma)$. On the other hand, if $G$ is invertible in $(C_\mu^+(\dot{\mathbb{R}}))^{2 \times 2}$, then $\det G= \alpha^2-q\delta^2=gg_* \in \mathcal{G}C_\mu^+(\dot{\mathbb{R}})$. Therefore $g$ is bounded away from zero and we conclude that $g=\mathcal{I}(G) \in \mathcal{G}C_{\mu}^{+}(\Gamma)$.

Conversely, if $g \in \mathcal{G}C_{\mu}^{+}(\Gamma)$, then $g_\mathcal{E}, \lambda_+^2g_\mathcal{O} \in C_\mu^+(\dot{\mathbb{R}})$ (cf. section 2) and $gg_* \in \mathcal{G}C_\mu^+(\dot{\mathbb{R}})$, and it follows from (\ref{6.21}) that $\mathcal{I}^{-1}(g)\in \mathcal{G}(C_{\mu}^{+}(\dot{\mathbb{R}}))^{2 \times 2}$.

We can prove analogously that the image of $C(Q)\cap \mathcal{G}(C_{\mu}^{-}(\dot{\mathbb{R}}))^{2 \times 2}$ by $\mathcal{I}$ is $\mathcal{G}C_{\mu}^{-}(\Gamma)$.
\end{proof}
We can now characterize completely the subclass of matrix functions belonging to $C(Q)$ and admitting a commutative canonical factorization within $C(Q)$ (see for instance \cite{CSBastos2, PS} as regards the discussion of this problem).
\begin{thm}
$G \in C(Q)$ admits a canonical WH factorization with $G_{\pm} \in C(Q)$ if and only if its $\Sigma$-symbol $g$ admits a special $\Sigma$-factorization.
\end{thm}
\begin{proof}
If $G=G_-G_+$ with $G_{\pm} \in C(Q) \cap \mathcal{G}(C_{\mu}^{\pm}(\dot{\mathbb{R}}))^{2 \times 2}$ then by \linebreak Proposition \ref{prop6.7A} we have $g=g_-g_+$ with $g_{\pm}\in \mathcal{G}C_{\mu}^{\pm}(\Gamma)$, and conversely.
\end{proof}
We remark that in the case where $G$ admits a canonical WH factorization we necessarily have $\ind_{\Gamma}g=0$, so that (\ref{4.8H}) gives a necessary and sufficient condition for existence of a factorization $G=G_{-}G_{+}$ with factors in $ C(Q)$.

It is also useful to remark at this point that, if $f \in \mathcal{G}C_{\mu}(\dot{\mathbb{R}})$ and we identify it with a function in $\mathcal{G}C_{\mu}(\Gamma)$, we have $\mathcal{I}^{-1}(f)=fI$. Moreover:
\begin{equation}
\hspace{-6cm}\mathcal{I}^{-1}(r_{\nu})=\left[
                      \begin{array}{cc}
                        \nu & \frac{\xi-i}{\xi+i} \\
                        \frac{\xi+ik_{0}}{\xi-ik_{0}} & \nu \\
                      \end{array}
                    \right]=:R_{\nu},
\label{6.22}\end{equation}
\begin{equation*}
\mathcal{I}^{-1}(\alpha_{-})=\left[
                      \begin{array}{cc}
                        C & \frac{\rho_{-}}{\xi+i} \\
                        \frac{\xi+ik_{0}}{\rho_{-}} & C \\
                      \end{array}
                    \right]=:A_{-}, \quad \mathcal{I}^{-1}(\alpha_{+})=\left[
                      \begin{array}{cc}
                        C & \frac{\xi-ik_{0}}{\rho_{+}} \\
                         \frac{\rho_{+}}{\xi-i} & C \\
                      \end{array}
                    \right]=:A_{+}
\label{6.23}\end{equation*}
\begin{equation*}
\hspace{-0.3cm}\mathcal{I}^{-1}(\alpha_{+}^{-1}(\alpha_{+})_*)=\frac{2}{k_{0}-1}\frac{\xi-i}{\xi+i}\left[
                      \begin{array}{cc}
                        C^2+\frac{\xi-ik_{0}}{\xi-i} & -2C\frac{\xi-ik_{0}}{\rho_{+}} \\
                        -2C\frac{\rho_{+}}{\xi-i} & C^2+\frac{\xi-ik_{0}}{\xi-i} \\
                      \end{array}
                    \right] \in \mathcal{G}C_{\mu}^+(\dot{\mathbb{R}})^{2 \times
                    2},
\label{6.24}\end{equation*} where $r_{\nu}$, $\alpha_{\pm}$ and $C$
are defined by (\ref{3.16}), (\ref{3.19}) and (\ref{3.20}),
respectively.

Since, by Theorem \ref{theo5.4}, every $g \in \mathcal{G}C_{\mu}(\Gamma)$ admits a meromorphic $\Sigma$-factorization of the form (\ref{5.2}), it is clear that applying $\mathcal{I}^{-1}$ to its right-hand side we obtain a meromorphic factorization (\cite{LS, CLSpeck}) for any $G \in C(Q)$.

Moreover, by (\ref{5.5.A}) and Theorem \ref{theo5.1}, we can reduce the problem of factorizing $g$ to the case where, apart from a rational function $(\lambda_-/\lambda_+)^{-\tilde{k}}$, we have
$g=g_-\alpha_-^s r_{\mu}^t\alpha_+^v g_+,$ with $s,t \in \{0,1,2\}$ and $v \in \{0,1\}$, $g_{\pm} \in \mathcal{G}C_\mu^{\pm}(\Gamma)$. In this case by applying $\mathcal{I}^{-1}$ we obtain a meromorphic factorization of the form
\begin{equation}\label{A}G=\mathfrak{G}_{-}\mathcal{M}\mathfrak{G}_{+}, \ \text{ where } \mathfrak{G}_{\pm}\in \mathcal{G}(C_{\mu}^{\pm}(\dot{\mathbb{R}}))^{2 \times 2}\end{equation}
 and the middle factor $\mathcal{M}$ is a product whose factors are equal to $A_{+}, A_{-}$ or $R_{\nu}$. More precisely, the middle factor $\mathcal{M}$ takes one of the forms $I, R_{\nu}, R_{\nu}^2$, $A_{-}A_{+}$, $A_{-}R_{\nu}A_{+}$ if $m=0$, or $A_{-}$, $A_{-}R_{\nu}$, $A_{-}R_{\nu}^2$, $A^2_{-}A_{+}$, $A_{-}^2R_{\nu}A_{+}$ if $m=1$. The following results, together with (\ref{6.21}) and Proposition \ref{prop6.7A}  provide a WH factorization for $\mathcal{M}$ in each case.
\begin{thm}\label{theo6.7}
If $\nu\neq 0$, $\nu^2\neq 1$ then $R_{\nu}$ admits a canonical WH factorization
$R_{\nu}=(R_{\nu})_{-}(R_{\nu})_{+}$ with
\begin{equation*}
(R_{\nu})_{-}=\left[
                      \begin{array}{cc}
                        \nu & 0 \\
                        \frac{\xi+ik_{0}}{\xi-ik_{0}} & \frac{1-\nu^2}{\nu}\frac{\xi-k_{0}/z_{0}}{\xi-ik_{0}} \\
                      \end{array}
                    \right], \quad (R_{\nu})_{+}=\left[
                      \begin{array}{cc}
                        1 & \frac{1}{\nu}\frac{\xi-i}{\xi+i} \\
                        0 & -\frac{\xi-z_{0}}{\xi+i} \\
                      \end{array}
                    \right].
\label{6.25}\end{equation*} If $\nu=
0$, then $R_{\nu}$ admits the non-canonical WH factorization \linebreak
$R_{0}=(R_{0})_{-}\diag (\lambda_{+}/\lambda_{-}, \lambda_{-}/\lambda_{+})(R_{0})_{+}$ where
\begin{equation}
(R_{0})_{-}=\left[
                      \begin{array}{cc}
                        0 & 1 \\
                        \frac{\xi-i}{\xi-ik_{0}} & 0 \\
                      \end{array}
                    \right], \quad (R_{0})_{+}=\left[
                      \begin{array}{cc}
                        \frac{\xi+ik_{0}}{\xi+i} & 0 \\
                        0 & 1 \\
                      \end{array}
                    \right].
\label{6.26}\end{equation}
\end{thm}
\begin{proof}The equality $R_{\nu}=(R_{\nu})_{-}(R_{\nu})_{+}$, for $\nu \neq 0$, $\nu^2 \neq 1$, can be checked directly, taking (\ref{3.18A}) into account, and it is easy to verify that \linebreak $(R_{\nu})_{\pm}\in (C_{\mu}^{\pm}(\dot{\mathbb{R}}))^{2 \times 2}$ and $\det (R_{\nu})_{\pm}\in \mathcal{G}C_{\mu}^{\pm}(\dot{\mathbb{R}})$. The factorization for $R_{0}$ is straightforward.
\end{proof}
As regards the statement of the previous theorem, we remark that in (\ref{3.15}) we may have $\nu=0$ but we never have $\nu^2=1$, so that the latter case was not considered above. On the other hand it is well known, and easy to see, that the canonical WH factorization presented in Theorem \ref{theo6.7} for $R_{\nu}$ is not unique; however, by choosing that particular factorization, we obtain factors which satisfy some relations that will be useful later. Namely, we have
\begin{equation}
(R_{\nu})_{+}\diag (1,\lambda_{+}/\lambda_{-})=\diag (1,\lambda_{+}/\lambda_{-})T_{+}
\label{6.26a}\end{equation}
where the second factor on the right hand side is a matrix function belonging to $\mathcal{G}C_{\mu}^{+}(\dot{\mathbb{R}})^{2 \times 2}$, given by
\begin{equation}
T_{+}=\left[
                      \begin{array}{cc}
                        1 & \frac{1}{\nu} \\
                         0 & -\frac{\xi-z_{0}}{\xi+i} \\
                      \end{array}
                    \right] \text{ if } \nu \neq 0,
\label{6.26b}\end{equation}
 $T_{+}=(R_{0})_{+}$ if $\nu=0$.

The following property regarding $R_{\nu}$ will also be used later.
\begin{lem}\label{lem6.9}
Let $R_{\nu}$ be given by $(\ref{6.22})$; we have
\begin{equation}
\diag (1,\lambda_{-}/\lambda_{+})R_{\nu}\diag (1,\lambda_{+}/\lambda_{-})=\left[
                      \begin{array}{cc}
                        \nu & 1 \\
                         \frac{(\xi-i)(\xi+ik_{0})}{(\xi+i)(\xi-ik_{0})} &  \nu \\
                      \end{array}
                    \right]=:\tilde{R}_{\nu},
\label{6.26c}\end{equation}
and $\tilde{R}_{\nu}$ admit a canonical WH factorization $\tilde{R}_{\nu}=(\tilde{R}_{\nu})_{-}(\tilde{R}_{\nu})_{+}$ with
\begin{equation}
(\tilde{R}_{\nu})_{-}=\left[
                      \begin{array}{cc}
                        1 & 0 \\
                         \nu & -(\nu^2-1)\frac{\xi-k_{0}/z_{0}}{\xi-ik_{0}}  \\
                      \end{array}
                    \right], \quad (\tilde{R}_{\nu})_{+}=\left[
                      \begin{array}{cc}
                        \nu & 1 \\
                         \frac{\xi-z_{0}}{\xi+i}  & 0  \\
                      \end{array}
                    \right].
\label{6.26d}\end{equation}
\end{lem}
\begin{proof}
 The equality in (\ref{6.26c}) is obvious and it is easy to verify  that \linebreak $\tilde{R}_{\nu}=(\tilde{R}_{\nu})_{-}(\tilde{R}_{\nu})_{+}$, taking the second equality of (\ref{3.18A}) into account. The relations $(\tilde{R}_{\nu})_{\pm}\in \mathcal{G}(C^{\pm}_{\mu}(\dot{\mathbb{R}}))^{2 \times 2}$ are also simple to check.
\end{proof}
\begin{thm} \label{theo6.10}
If $\nu \neq 0$ (and $\nu^2 \neq 1$), $R_{\nu}^2$ admits a canonical WH factorization $R_{\nu}^2=(R^2_{\nu})_{-}(R^2_{\nu})_{+}^{-1}$ with $(R^2_{\nu})_{\pm}=[r_{ij}^{\pm}]$ given by
\begin{itemize}
  \item []$r_{11}^+=\frac{1}{\nu}\frac{\xi-i}{\xi-z_{0}}-\frac{1}{\nu}\frac{\xi-i}{\xi+i}r_{21}^+$
  \item []$r_{21}^+=-\frac{1}{\nu^2-1}\frac{(\xi+i)^2}{(\xi-z_{0})(\xi-k_{0}/z_{0})}\left[\frac{\xi-ik_{0}}{\xi-z_{0}}(\nu^2+\frac{(\xi-i)(\xi+ik_{0})}{(\xi+i)(\xi-ik_{0})})-B\nu\right]$
 \item []$r_{12}^+=\frac{1}{\nu}\frac{\xi+i}{\xi-z_{0}}-\frac{1}{\nu}\frac{\xi-i}{\xi+i}r_{22}^+$
  \item []$r_{22}^+=\tilde{B}\frac{\nu}{\nu^2-1}(\frac{\xi+i}{\xi-z_{0}})^2$
   \item []$r_{11}^-=B\frac{\xi-i}{\xi-ik_{0}}$
  \item []$r_{21}^-=\frac{B\nu(\xi+i)-(\nu^2-1)(\xi-k_{0}/z_{0})}{\xi-ik_{0}}$,
  \item []
  $r_{12}^-=\frac{1}{\nu}\frac{\xi+i}{\xi-z_{0}}\left(\nu \tilde{B}\frac{(\xi-i)(\xi-k_{0}/z_{0})}{(\xi+i)(\xi-ik_{0})}+\nu^2+\frac{(\xi-i)(\xi+ik_{0})}{(\xi+i)(\xi-ik_{0})}\right)$
  \item []$r_{22}^-=\nu\frac{\xi+i}{\xi-i}r_{12}^- - (\nu^2-1)\frac{(\xi+i)(\xi-k_{0}/z_{0})}{(\xi-i)(\xi-ik_{0})}=\frac{\xi+i}{\xi-z_{0}}\left(2\frac{\xi+ik_{0}}{\xi-ik_{0}}+\tilde{B}\nu\frac{\xi-k_{0}/z_{0}}{\xi-ik_{0}}\right)$,
      \end{itemize}
 where  $$\hspace{-1.5cm}B=\frac{1}{\nu}\left[\frac{\xi-ik_{0}}{\xi-z_{0}}\left(\nu^2+\frac{(\xi-i)(\xi+ik_{0})}{(\xi+i)(\xi-ik_{0})}\right)\right]_{\xi=k_{0}/z_{0}},
            $$ $$\tilde{B}=-\frac{1}{\nu}\left[\frac{(\xi+i)(\xi-ik_{0})}{(\xi-i)(\xi-k_{0}/z_{0})}\left(\nu^2+\frac{(\xi-i)(\xi+ik_{0})}{(\xi+i)(\xi-ik_{0})}\right)\right]_{\xi=z_{0}}.
            $$
\end{thm}
\begin{proof}
From Corollary \ref{cor6.3}, $R_{\nu}^2$ admits a canonical WH factorization. We have
\begin{equation*}\label{K}
R_{\nu}^2=\left[
            \begin{array}{cc}
              \nu^2+\frac{(\xi-i)(\xi+ik_{0})}{(\xi+i)(\xi-ik_{0})} & 2\nu\frac{\xi-i}{\xi+i} \\
              2\nu\frac{\xi+ik_{0}}{\xi-ik_{0}} & \nu^2+\frac{(\xi-i)(\xi+ik_{0})}{(\xi+i)(\xi-ik_{0})} \\
            \end{array}
          \right]=M_{-}M_{+}
\end{equation*}
where \begin{equation*}\label{L}M_{-}=\left[
              \begin{array}{cc}
                \nu & \frac{\xi-i}{\xi-ik_{0}} \\
                \frac{\xi+ik_{0}}{\xi-ik_{0}} & \nu\frac{\xi+i}{\xi-ik_{0}} \\
              \end{array}
            \right], \quad M_{+}=\left[
              \begin{array}{cc}
                \nu & \frac{\xi-i}{\xi+i} \\
                \frac{\xi+ik_{0}}{\xi+i} & \nu\frac{\xi-ik_{0}}{\xi+i} \\
              \end{array}
            \right],\end{equation*}
            so that the equation $R_{\nu}^2\phi_{+}=\phi_{-}$, $\phi_{\pm} \in (C_{\mu}^{\pm})^2$ is equivalent to
\begin{equation}\label{M}
M_{+}\phi_{+}=M_{-}^{-1}\phi_{-}.
\end{equation}
Solving (\ref{M}) under the condition $\phi_{1+}(i)=0$, $\phi_{2+}(i)\neq 0$, we obtain \linebreak$\phi_{+}=(r_{11}^+, r_{21}^+)$, $\phi_{-}=(r_{11}^-, r_{21}^-)$; solving the same equation under the condition $\phi_{1-}(-i)\neq 0$, $\phi_{2-}(-i)= 0$, we obtain $\phi_{+}=(r_{12}^+, r_{22}^+)$, $\phi_{-}=(r_{12}^-, r_{22}^-)$. Thus we have $R_{\nu}^2(R_{\nu}^2)_{+}=(R_{\nu}^2)_{-}$ and, since $(R_{\nu}^2)_{-}(-i)$ is invertible, we conclude that $R_{\nu}^2=(R_{\nu}^2)_{-}(R_{\nu}^2)_{+}^{-1}$ is a canonical WH factorization (\cite{CSBastos2}, \linebreak Theorem 3.1).
\end{proof}
\begin{rem}\label{rem6.11}The canonical WH factorization of $R_{\nu}^2$ where $\nu=0$ can be obtained trivially since $R_{0}^2=[(\xi-i)(\xi+ik_{0})]/[(\xi+i)(\xi-ik_{0})]I$.

\end{rem}
The factors involved in the factorizations presented in Theorem \ref{theo6.10} and Remark \ref{rem6.11} also possess some useful properties. We have, in particular,
\begin{equation}\label{6.32A}
\diag (1,\lambda_{-}/\lambda_{+})(R_{\nu}^2)_{-}=\tilde{R}_{-}\diag(\lambda_{-}/\lambda_{+},1)
\end{equation}
where
\begin{equation}\label{6.32B}
\tilde{R}_{-}=\left[
                \begin{array}{cc}
                  B\frac{\xi+i}{\xi-ik_{0}} & r_{12}^{-} \\
                  r_{21}^{-} & \nu r_{12}^{-}-(\nu^2-1)\frac{\xi-k_{0}/z_{0}}{\xi-ik_{0}} \\
                \end{array}
              \right]\in \mathcal{G}(C^{-}_{\mu}(\dot{\mathbb{R}}))^{2 \times 2}.
\end{equation}
Now we consider the factorization of the non-rational matrices $A_{\pm}$.
\begin{thm} \label{theo6.11}$A_{+}$ and $A_{-}$ admit the following non-canonical WH factorizations:
\begin{equation}\label{O}
A_{-}=\tilde{A}_{-}\diag (1,\lambda_{-}/\lambda_{+}), \quad A_{+}=\diag (1,\lambda_{+}/\lambda_{-})\tilde{A}_{+},
\end{equation}
with $\tilde{A}_{\pm} \in \mathcal{G}(C_{\mu}^{\pm}(\dot{\mathbb{R}}))^{2 \times 2}$ given by
\begin{equation}\label{P}
\tilde{A}_{-}=\left[
                \begin{array}{cc}
                  C & \frac{\xi-ik_{0}}{\rho_{-}} \\
                  \frac{\xi+ik_{0}}{\rho_{-}} & C\frac{\xi+i}{\xi-i} \\
                \end{array}
              \right], \quad \tilde{A}_{+}=\left[
                \begin{array}{cc}
                  C & \frac{\xi-ik_{0}}{\rho_{+}} \\
                  \frac{\rho_{+}}{\xi+i} & C\frac{\xi-i}{\xi+i} \\
                \end{array}
              \right].
\end{equation}
\end{thm}
\begin{proof}The equalities in (\ref{O}), (\ref{P}) can be easily verified; on the other hand, it is clear that $\tilde{A}_{\pm} \in (C_{\mu}^{\pm}(\dot{\mathbb{R}}))^{2 \times 2}$ and $\det \tilde{A}_{\pm}=(k_{0}-1)/2 \in \mathbb{C}\backslash \{0\}$.
\end{proof}
Analogously to what happened in the previous factorizations, the factors $\tilde{A}_{\pm}$ possess some properties which will later be helpful.
\begin{lem}\label{lem6.13}
For $\tilde{A}_{-}$ defined by (\ref{P}) we have  $\diag (1,\lambda_{-}/\lambda_{+})\tilde{A}_{-}\linebreak =B_{-}\diag (\lambda_{-}/\lambda_{+},1)$
 where
\begin{equation}\label{P2}
B_{-}=\left[
          \begin{array}{cc}
            C\frac{\xi+i}{\xi-i} & \frac{\xi-ik_{0}}{\rho_{-}} \\
            \frac{\xi+ik_{0}}{\rho_{-}} & C \\
          \end{array}
        \right] \in \mathcal{G}(C_{\mu}^-(\dot{\mathbb{R}}))^{2 \times 2}
\end{equation}
\end{lem}
\begin{proof}Straightforward.
\end{proof}
Now we can present WH factorizations for the middle factor $\mathcal{M}$ in (\ref{A}) when it is not of the form $I, R_{\nu}$ (see Theorem \ref{theo6.7}), $R_{\nu}^2$ (see Theorem \ref{theo6.10} and Remark \ref{rem6.11}) or $A_{-}$ (see Theorem \ref{theo6.11}).
 \begin{thm} We have the following WH factorizations:
\begin{equation}\label{Q}
\hspace{-4,3cm}A_{-}A_{+}=\tilde{A}_{-}\tilde{A}_{+},
\end{equation}
\begin{equation}\label{Q1}
\hspace{-1,7cm}A_{-}R_{\nu}A_{+}=(\tilde{A}_{-}(\tilde{R}_{\nu})_{-})((\tilde{R}_{\nu})_{+}\tilde{A}_{+})
\end{equation} with $(\tilde{R}_{\nu})_{\pm}$ given by (\ref{6.26d}),
\begin{equation}\label{Q2}
\hspace{-1,7cm}A_{-}R_{\nu}^2=(\tilde{A}_{-}\tilde{R}_{-})\diag (\lambda_{-}/\lambda_{+},1)(R^2_{\nu})_{+}
\end{equation} with $\tilde{R}_{-}$ given by (\ref{6.32B}),
 \begin{equation}\label{R}
\hspace{-2cm}A^2_{-}A_{+}
=(\tilde{A}_{-}B_{-})\diag (\lambda_{-}/\lambda_{+},1)\tilde{A}_{+}
\end{equation}with $B_{-}$ given by (\ref{P2}),
\begin{equation}\label{R1}
A^2_{-}R_{\nu}A_{+}
=(\tilde{A}_{-}B_{-}(R_{\nu})_{-})\diag (\lambda_{-}/\lambda_{+},1)(T_{+}\tilde{A}_{+})
\end{equation} with $T_{+}$ given by (\ref{6.26b}), and $A_{-}R_{\nu}
=(\tilde{A}_{-}(\tilde{R}_{\nu})_{-}Q_{-})\diag (1,\lambda_{-}/\lambda_{+})Q_{+}^{-1}$
with
\begin{equation}\label{T2}
Q_{+}=\left[
          \begin{array}{cc}
            \frac{\xi+i}{\xi-z_{0}} & \frac{\xi+i}{\xi-z_{0}} \\
            \varrho \frac{\xi+i}{\xi-i}-\nu\frac{(\xi+i)^2}{(\xi-i)(\xi-z_{0})} & (\delta-\nu\frac{\xi+i}{\xi-z_{0}})\frac{\xi+i}{\xi-i}-1 \\
          \end{array}
        \right]
\end{equation}

\vspace{0,2cm}

\begin{equation*}\label{P3}
\hspace{-4,2cm}Q_{-}=\left[
          \begin{array}{cc}
            \varrho & \varrho\frac{\xi+i}{\xi-i}-1 \\
            1 & \frac{\xi+i}{\xi-i} \\
          \end{array}
        \right],
\end{equation*} where \begin{equation*}\label{P4}\hspace{-3,9cm}\varrho =\nu (\frac{\xi+i}{\xi-z_{0}})_{\xi=i}=\frac{2i\nu}{i-z_{0}}.\end{equation*}
\end{thm}
\begin{proof}
The canonical factorization (\ref{Q}) is a direct consequence of Theorem \nolinebreak \ref{theo6.11}. From the latter theorem and Lemma \ref{lem6.9}, we obtain (\ref{Q1}). The factorization in (\ref{Q2}) follows from Theorems \ref{theo6.11} and \ref{theo6.10}, Remark \ref{rem6.11} and (\ref{6.32A}). Theorem \ref{theo6.11} and Lemma \ref{lem6.13} imply (\ref{R}), while Theorem \ref{theo6.11} and (\ref{6.26a}) imply (\ref{R1}). Finally, we have  $A_{-}R_{\nu}=\tilde{A}_{-}\diag (1, \lambda_{-}/\lambda_{+})R_{\nu}$ where, by Lemma \ref{lem6.9},
$$\diag (1, \lambda_{-}/\lambda_{+})R_{\nu}=\tilde{R}_{\nu}\diag (1, \lambda_{-}/\lambda_{+})=$$
$$=(\tilde{R}_{\nu})_{-}(\tilde{R}_{\nu})_{+}\diag (1, \lambda_{-}/\lambda_{+})=
(\tilde{R}_{\nu})_{-}\left[
                         \begin{array}{cc}
                           \nu & \frac{\xi-i}{\xi+i} \\
                           \frac{\xi-z_{0}}{\xi+i} & 0 \\
                         \end{array}
                       \right]
$$ and $$\left[
                         \begin{array}{cc}
                           \nu & \frac{\xi-i}{\xi+i} \\
                           \frac{\xi-z_{0}}{\xi+i} & 0 \\
                         \end{array}
                       \right]=Q_{-}\diag (1, \lambda_{-}/\lambda_{+})Q_{+}^{-1}$$
with $Q_{\pm}$ defined by (\ref{T2})-(\ref{P4}).\end{proof}
As an illustration of the application of the previous results, we present the examples that follow.

\vspace{0.1cm}

\emph{\underline{Example 1.}} We consider here the factorization problem for $G \in C(Q)$ of the form
\begin{equation}
G=\exp (tL)
\label{6.31}
\end{equation}
where $t$ is a real parameter and $L$ is a rational matrix function
\begin{equation}
L=\left[
    \begin{array}{cc}
      0 & \frac{\xi-ik_{0}}{\xi+i} \\
      \frac{\xi+ik_{0}}{\xi-i}  & 0 \\
    \end{array}
  \right].
\label{6.32}
\end{equation}
This can be seen as the real line analogue of a factorization problem relative to the unit circle $\mathcal{S}^{1}$, arising when solving a Lax equation for some integrable systems (\cite{CSS1, Reyman}). We assume here for simplicity that $G$ takes the normal form (\ref{6.6}); for a general $L \in \mathcal{GR}^{2 \times 2}$, $G$ defined by (\ref{6.31}) can be reduced to the normal form by multiplication on the left and on the right by a rational matrix and its inverse, respectively (\cite{CMalheiro}). For $L$ given by (\ref{6.32}) we have
$L=(\xi^2+1)^{-1}S^{-1}\diag (\rho, -\rho)S$
where
\begin{equation}
S=\left[
    \begin{array}{cc}
      1 & \frac{\rho}{\mathfrak{p}_{1}} \\
     1  & -\frac{\rho}{\mathfrak{p}_{1}} \\
    \end{array}
  \right],
\label{6.33}
\end{equation}
so that $G$ can be diagonalized with eigenvalues $g_{1}=\exp (t\rho(\xi^2+1)^{-1})$, $g_{2}=\exp (-t\rho(\xi^2+1)^{-1})$, for which $\ind g_{1}=\ind g_{2}=0$. An important question when studying that kind of factorization problem is to determine for which values of the (dynamical variable) $t$ does $G$ admit a canonical WH factorization, which is connected with the question of global existence of solutions to some Lax equations (see for instance \cite{CSS1, Reyman}). We have the following.
\begin{thm}
 $G$ admits a canonical bounded factorization for all $t \in \mathbb{R}$.
 \end{thm}
 \begin{proof}
 Since $\det G =1$ $(m=0)$, $G$ admits a canonical WH factorization (which is necessarily bounded) if and only if the RH problem (\ref{6.10.1}) admits only the trivial solution $\phi_{\pm}=0$. By Theorem \ref{theo6.1} there are non-zero solutions to that problem if and only if $\beta_{g}=iK' \mod L$, where $g=\exp (t\tau(\xi^2+1)^{-1})$ is the $\Sigma$-symbol of $G$. Since
 $$\beta_{g}=\frac{k_{0}}{2\pi}\int_{\Gamma}\frac{t}{\xi^2+1}d\xi=k_{0}t \in \mathbb{R},$$
 we conclude that we must have $\phi_{\pm}=0.$
 \end{proof}
The factorization of $G$ can be obtained in each case (depending on the value of $t$) from Theorem \ref{theo5.1}, the properties of $\mathcal{I}$ and the preceding results in this section. In particular we conclude that $G$ admits a factorization $G=G_{-}G_{+}$ with $G_{\pm}$ in $C(Q)$ if and only if $k_{0}t=0 \mod L$ and, assuming that $g=g_{-}g_{+}$ is a special $\Sigma$-factorization in that case, the factors are \linebreak $G_{\pm}=\mathcal{I}^{-1}(g_{\pm})$.

\vspace{0.1cm}

\emph{\underline{Example 2.}}
 Let $G \in C(Q)$ and let $g$ be its $\Sigma$-symbol. We consider here two cases related, on the one hand, to Theorem \ref{theo6.2} and  Corollary \ref{cor6.3} and, on the other hand, to Theorem \ref{theo6.5}.

 In the first case, suppose that the assumptions of Theorem \ref{theo6.2} hold. Then it follows from (\ref{6.13}) that $G=\mathcal{I}^{-1}(g_{-})\mathcal{I}^{-1}(r_{0})\mathcal{I}^{-1}(g_{+})$ (with \linebreak $r_{0}=\tau[(\xi+i)(\xi-ik_{0})]^{-1}$). Then, from (\ref{6.21}), (\ref{6.22}) and (\ref{6.26}) we have $G=G_{-}DG_{+}$ with $D= \diag (\lambda_{+}/\lambda_{-}, \lambda_{-}/\lambda_{+})$,
 $$G_{-}=\left[
           \begin{array}{cc}
             (\xi-i)^2(g_{-})_{\mathcal{O}} & (g_{-})_{\mathcal{E}} \\
             \frac{\xi-i}{\xi-ik_{0}}(g_{-})_{\mathcal{E}} & (g_{-})_{\mathcal{O}}\mathfrak{p}_{1}\\
           \end{array}
         \right], G_{+}=\left[
           \begin{array}{cc}
             \frac{\xi+ik_{0}}{\xi+i}(g_{+})_{\mathcal{E}} &  \frac{\xi+ik_{0}}{\xi+i}(g_{+})_{\mathcal{O}} \mathfrak{p}_{2}\\
            (g_{+})_{\mathcal{O}}\mathfrak{p}_{1}& (g_{+})_{\mathcal{E}} \\
           \end{array}
         \right]
 $$
 As a result, the factorization of $G$ allows to determine two linearly independent solutions to (\ref{6.12}) with $\psi_{\pm}$ in $C_{\mu}^{\pm}(\Gamma)$ (see Remark \ref{rmk6.4}). Denoting by $G_{1+}$ and $G_{1-}$ the first column of $G_{+}^{-1}$ and $G_{-}$, respectively, those solutions are $(T_{\Sigma}(SG_{1+}), T_{\Sigma}(\lambda_{+}\lambda_{-}^{-1}SG_{1-}) )$ and $(T_{\Sigma}(\lambda_{-}\lambda_{+}^{-1}SG_{1+}), T_{\Sigma}(SG_{1-}) )$ where $S$ was defined in (\ref{6.33}).

 In the second case, suppose that the assumptions of Theorem \ref{theo6.5} hold and, for simplicity, $\ind g_{1}=0$, $\ind g_{2}=1$ and $\tilde{g}=g\alpha_{+}$ admits a special $\Sigma$-factorization $\tilde{g}=\tilde{g}_{-}\tilde{g}_{+}$. Then a WH factorization for $G$, with partial indices 0, 1 as in Theorem \ref{theo6.5}, is $G=G_{-}DG_{+}$ with
 $D= \diag (\lambda_{-}/\lambda_{+}, 1)$,
 $$G_{-}=\left[
           \begin{array}{cc}
             (\tilde{g}_{-})_{\mathcal{E}} & \mathfrak{p}_{2}(\tilde{g}_{-})_{\mathcal{O}} \\
             \mathfrak{p}_{1}(\tilde{g}_{-})_{\mathcal{O}} & (\tilde{g}_{-})_{\mathcal{E}} \\
           \end{array}
         \right], G_{+}=\frac{2}{k_{0}-1}\tilde{J}\tilde{A}_{+}\tilde{J}\left[
           \begin{array}{cc}
             (\tilde{g}_{+})_{\mathcal{E}} &  \mathfrak{p}_{2}(\tilde{g}_{+})_{\mathcal{O}} \\
             \mathfrak{p}_{1}(\tilde{g}_{+})_{\mathcal{O}} & (\tilde{g}_{+})_{\mathcal{E}} \\
           \end{array}
         \right]
 $$
 where we took into account that $\mathcal{I}^{-1}(\alpha_{+}^{-1})=A_{+}^{-1}=2(k_0-1)^{-1}\diag (\lambda_{-}/\lambda_{+}, 1)\linebreak\tilde{J}\tilde{A}_{+}\tilde{J}$, with  $\tilde{J}=\diag (-1,1)$.

\vspace{0.4cm}

\textbf{Acknowledgments}

\vspace{0.2cm}

The work on this paper was partially supported by Funda\c{c}\~{a}o para a Ci\^{e}ncia e a Tecnologia
(FCT) through the project PTDC/MAT/81385/2006 and the Program POCI 2010/FEDER.

\end{document}